\documentclass[pdflatex,sn-mathphys-num]{sn-jnl}

\usepackage{graphicx}%
\usepackage{multirow}%
\usepackage{amsmath,amssymb,amsfonts}%
\usepackage{amsthm}%
\usepackage[title]{appendix}%
\usepackage{xcolor}%
\usepackage{textcomp}%
\usepackage{manyfoot}%
\usepackage{booktabs}%
\usepackage{algorithm}%
\usepackage{algorithmicx}%
\usepackage{algpseudocode}%
\usepackage{listings}%

\theoremstyle{thmstyleone}%
\newtheorem{theorem}{Theorem}
\newtheorem{proposition}[theorem]{Proposition}%

\theoremstyle{thmstyletwo}%

\theoremstyle{thmstylethree}%
\newtheorem{definition}{Definition}%

\newtheorem{assumption}{Assumption}
\newtheorem{corollary}{Corollary}
\newtheorem{lemma}{Lemma}
\begin{document}

\title[Article Title]{Structural Compatibility and Uniform Stability of Temporally Degenerate Parabolic Systems}


\author[]{\fnm{Amadou} \sur{Cissé}}\email{amadou.cisse@univ-lorraine.fr}


\affil[]{\orgdiv{CRAN, CNRS, UMR 7039}, \orgname{University of Lorraine}, \orgaddress{
\city{Metz}, \postcode{F-57000}, 
\country{France}}}





\abstract{Modern feedback design for distributed parameter systems presupposes that the closed-loop dynamics define a well-posed evolution problem.
This presupposition becomes nontrivial for temporally degenerate parabolic systems, where temporal degeneracy affects not only the analytical properties of the evolution equation but also the mathematical formulation of the feedback interconnection itself.
It is shown that admissible feedback interconnections for temporally degenerate parabolic systems are completely characterized by an operator compatibility condition linking the singular reaction operator with the actuator and observation operators. This characterization removes the singular component of the closed-loop dynamics and reduces the degenerate evolution equation to a regular evolution equation.
Building upon this regularized formulation, a critical--residual decomposition yields a uniform exponential stability certificate, which is subsequently extended to the original infinite-dimensional evolution through a finite-to-infinite lifting theorem. A constructive static output feedback synthesis is finally obtained as a consequence of these results. Numerical experiments illustrate the regularization mechanism, validate the stability certificate, and confirm the finite-to-infinite lifting.}


\maketitle

\section{Introduction}

The past decades have witnessed remarkable advances in the mathematical theory of infinite-dimensional evolution equations, providing the rigorous foundation on which modern feedback design for distributed parameter systems is formulated. The evolution-family theory established by Acquistapace and Terreni
\cite{AcquistapaceTerreni1984Existence,
AcquistapaceTerreni1985ConstantDomains,
AcquistapaceTerreni1987Unified}
laid the mathematical foundations for the analysis of non-autonomous evolution problems. In the context of degenerate parabolic equations, Ivasyshen and Medynsky~\cite{ivasyshen2000properties} investigated potential-type integral representations for parabolic systems degenerating on the initial hyperplane, while Ivasyshen and Voznyak~\cite{ivasyshen2000fundamental} studied fundamental solutions of the Cauchy problem for a class of degenerate parabolic equations. In parallel, the operator-theoretic formulation developed by Curtain and Zwart~\cite{CurtainZwart1995InfiniteDimensional} and systematically developed by Jacob and Zwart~\cite{JacobZwart2018OperatorTheoretic} established a unified mathematical setting for the analysis and control of distributed parameter systems. In this well-established setting, the synthesis of feedback laws is naturally formulated under the implicit assumption that the resulting closed-loop dynamics define a well-posed evolution problem.

This implicit assumption is fundamentally challenged when the coefficient multiplying the time derivative is allowed to vanish or become singular. In this case, the closed-loop dynamics may no longer generate a regular evolution process, thereby invalidating one of the basic assumptions underlying classical control design. Degenerate evolution equations have been studied from an abstract analytical perspective by Favini and Yagi~\cite{FaviniYagi1999Degenerate}, while a substantial control literature has developed for parabolic equations with spatially degenerate diffusion. In particular, Carleman estimates, controllability, and stabilization results have been established by Cannarsa and collaborators
\cite{CannarsaMartinezVancostenoble2008Carleman,
CannarsaMartinezVancostenoble2009Boundary}, and more recent works have developed Fredholm and backstepping stabilization techniques for spatially degenerate parabolic equations
\cite{GagnonLissyMarx2021Fredholm,
LissyMoreno2024Backstepping}.
These results demonstrate the substantial progress achieved for degenerate parabolic control problems. The situation considered here is structurally different: the degeneracy acts on the time derivative itself and may therefore affect the admissibility of the closed-loop evolution before stability can be addressed.

Against this mathematical background, the stabilization of distributed parameter systems has followed several complementary research directions. The literature encompasses spectral feedback techniques \cite{Triggiani1980Riccati,
BadraTakahashi2014Stabilization}, finite-dimensional compensation methods, reduced-order approximations, and spillover-aware stabilization strategies
\cite{Curtain1984FiniteDimensional,
CurtainSalamon1986Compensators,
CurtainGlover1986Robust,
Ito1990Galerkin,
GruneMeurer2022SmallGain,
TrelatWangXu2024POD}.
More recently, this landscape has been enriched by static output feedback methodologies for distributed parameter systems
\cite{GahlawatPeet2016SOF,
WuZhang2020StaticOutput,
LhachemiPrieur2025SOF}.
Despite their methodological diversity, these approaches share a common mathematical premise: controller synthesis is carried out on a closed-loop evolution model whose well-posedness is taken for granted. Consequently, the structural admissibility of the feedback interconnection is assumed rather than established, leaving open the question of whether such an interconnection can always be defined.

This observation raises a more fundamental question. Before addressing stabilization itself, one must determine whether a feedback interconnection compatible with the degenerate dynamics can be defined.
This issue is independent of any particular synthesis methodology and precedes questions of stability, optimality, or robustness. Instead, it concerns the conditions under which a feedback law gives rise to a mathematically admissible closed-loop evolution problem. The mathematical formulation of the control problem must therefore be revisited before controller synthesis itself can be systematically addressed.

The analysis establishes a structural principle governing the admissibility of feedback interconnections for temporally degenerate parabolic systems. Admissibility is characterized by an operator compatibility condition linking the intrinsic dynamics of the evolution equation with its input--output operators. This characterization transforms the original degenerate closed-loop problem into a regularized evolution equation, thereby recovering the mathematical setting required for infinite-dimensional control analysis. The resulting formulation provides the basis for a critical--residual decomposition leading to a uniform exponential stability certificate and a structural finite-to-infinite lifting theorem. A constructive static output feedback realization is then obtained as a consequence of this analysis. The theoretical developments are finally corroborated by numerical simulations validating the compatibility principle, the regularization identity, and the critical--residual stability mechanism.

The paper is organized as follows. Section~\ref{sec:problem}
introduces the class of temporally degenerate parabolic systems under consideration and formulates the underlying control problem.
Section~\ref{sec:compatibility} establishes the structural
compatibility principle, and Section~\ref{sec:factorization} provides its complete operator-theoretic characterization.
Section~\ref{sec:regular-evolution} analyzes the resulting regularized evolution equation.
Sections~\ref{sec:critical-residual} and \ref{sec:structural-lifting} develop the critical--residual stability theory and the associated finite-to-infinite lifting result, respectively.
Section~\ref{sec:synthesis} presents a constructive static output feedback realization.
Finally, Section~\ref{sec:numerics} provides a numerical validation of the principal theoretical results.

\section{Mathematical Framework and Problem Formulation}
\label{sec:problem}

This section formulates the mathematical setting and the feedback problem that serve as the foundation for the subsequent analysis.

\subsection{Abstract Evolution Equation}
\label{subsec:abstract-framework}

The analysis is developed within the following abstract evolution formulation.

Let \(X\) be a Hilbert space endowed with inner product \(\langle\cdot,\cdot\rangle_X\) and associated norm \(\|\cdot\|_X\).
The control and observation spaces are
\[
X_u=\mathbb{R}^{m},
\qquad
X_y=\mathbb{R}^{p}.
\]

The class of systems under consideration is described by
\begin{equation}
\alpha(t)\dot z(t)
+
\beta(t)A_0z(t)
+
A_1z(t)
+
\alpha(t)A_rz(t)
=
Bu(t),
\qquad
t\in(0,T],
\label{eq:abstract-system}
\end{equation}
together with the output equation
\begin{equation}
y(t)=Cz(t),
\label{eq:abstract-output}
\end{equation}
where
\[
z(t)\in X,
\qquad
u(t)\in X_u,
\qquad
y(t)\in X_y.
\]

The principal operator
\[
A_0:D(A_0)\subset X\rightarrow X
\]
describes the dominant parabolic dynamics. The bounded operator
\[
A_r\in\mathcal L(X)
\]
collects regular lower-order contributions, while the bounded operator
\[
A_1\in\mathcal L(X)
\]
models the reaction dynamics. The input and output operators are
\[
B\in\mathcal L(X_u,X),
\qquad
C\in\mathcal L(X,X_y),
\]
respectively.

Unlike the remaining operators, \(A_1\) is not multiplied by the temporal coefficient \(\alpha(t)\).
This structural asymmetry is the distinctive feature of the evolution equation and will play a central role in the subsequent analysis.

The scalar functions
\[
\alpha,\beta:(0,T]\rightarrow(0,\infty)
\]
describe the temporal modulation of the dynamics. The coefficient \(\alpha(t)\) weights the time derivative, whereas \(\beta(t)\) weights the principal parabolic operator. Their analytical properties are introduced in Subsection~\ref{subsec:analytical-setting}.

The above formulation encompasses a broad class of distributed-parameter systems. The next subsection introduces a representative temporally degenerate parabolic realization together with its closed-loop formulation.

\subsection{Degenerate Parabolic Model and Closed-Loop Reformulation}
\label{subsec:parabolic-model}

The following temporally degenerate parabolic model provides a representative realization of the abstract evolution equation.

Let
\[
\Omega\subset\mathbb{R}^{n}
\]
be a bounded domain with boundary \(\partial\Omega\)
of class \(C^{2}\).
The state evolution is governed by
\begin{equation}
\alpha(t)\partial_t z(t,x)
-
a^{2}\beta(t)\Delta z(t,x)
+
a_{0}z(t,x)
=
f(t,x),
\qquad
(t,x)\in(0,T]\times\Omega,
\label{eq:parabolic-model}
\end{equation}
where
\[
a>0,
\qquad
a_{0}\ge0,
\]
and
\[
\alpha,\beta\in C^{1}((0,T])
\]
are positive on
\((0,T]\).

The dynamics are supplemented with the homogeneous Dirichlet boundary condition
\begin{equation}
z(t,x)=0,
\qquad
(t,x)\in(0,T]\times\partial\Omega,
\label{eq:boundary-condition}
\end{equation}
and the initial condition
\begin{equation}
z(0,x)=z_{0}(x),
\qquad
x\in\Omega.
\label{eq:initial-condition}
\end{equation}

The distributed input is
\[
f(t,\cdot)=Bu(t),
\]
and the measured output is given by~\eqref{eq:abstract-output}.

Setting
\[
X=L^{2}(\Omega),
\]
the above model is recovered from \eqref{eq:abstract-system}
by choosing
\[
A_{0}
=
-a^{2}\Delta,
\qquad
D(A_{0})
=
H^{2}(\Omega)\cap H_{0}^{1}(\Omega),
\]
and
\[
A_{1}=a_{0}I,
\qquad
A_{r}=0.
\]

The abstract evolution equation nevertheless accommodates additional regular lower-order dynamics through the operator \(A_r\).
The present model therefore corresponds to the particular case \(A_r=0\).

The control law is chosen as
\begin{equation}
u(t)
=
\left(
K_{0}
+
\alpha(t)K_{1}
\right)
y(t),
\label{eq:control-law}
\end{equation}
where
\[
K_{0},K_{1}\in\mathbb R^{m\times p}
\]
are constant gain matrices.

Substituting \eqref{eq:control-law}
into \eqref{eq:abstract-system}
yields
\begin{equation}
\alpha(t)\dot z
+
\beta(t)A_{0}z
+
\left(
A_{1}-BK_{0}C
\right)z
+
\alpha(t)
\left(
A_{r}-BK_{1}C
\right)z
=
0.
\label{eq:closed-loop-before-desingularization}
\end{equation}

Since
\[
\alpha(t)>0,
\qquad
t>0,
\]
division by \(\alpha(t)\) gives the equivalent desingularized evolution equation. Defining
\begin{equation}
d(t)
=
\frac{\beta(t)}{\alpha(t)},
\label{eq:d-definition}
\end{equation}
one obtains
\begin{equation}
\dot z
+
d(t)A_{0}z
+
\frac{1}{\alpha(t)}
\left(
A_{1}-BK_{0}C
\right)z
+
\left(
A_{r}-BK_{1}C
\right)z
=
0.
\label{eq:closed-loop-singular}
\end{equation}

Equation~\eqref{eq:closed-loop-singular} is the fundamental object analyzed throughout the remainder of the paper.

\subsection{Analytical Setting and Problem Formulation}
\label{subsec:analytical-setting}

The analysis of \eqref{eq:closed-loop-singular} is carried out under the following assumptions.

The first assumption specifies the temporal coefficients governing the degenerate dynamics.

\begin{assumption}[Temporal coefficients]
\label{ass:time}

The functions \(\alpha,\beta\in C^{1}((0,T])\)
satisfy
\[
\alpha(t)>0,
\qquad
\beta(t)>0,
\qquad
t\in(0,T].
\]

Moreover,
\(\alpha(0)=0\),
and the ratio
\[
d(t):=\frac{\beta(t)}{\alpha(t)}
\]
satisfies
\begin{equation}
0<d_-\le d(t)\le d_+,
\qquad
t\in(0,T],
\label{eq:d-bounds}
\end{equation}
for some constants \(d_->0\) and \(d_+<\infty\).

\end{assumption}

The second assumption concerns the principal parabolic operator.

\begin{assumption}[Principal operator]
\label{ass:A0}
The operator \(A_0:D(A_0)\subset X\rightarrow X\) is densely defined, self-adjoint, strictly positive, and possesses a compact resolvent.
\end{assumption}

In the abstract evolution equation, the singular operator
\[
\frac{1}{\alpha(t)}
\left(
A_{1}-BK_{0}C
\right)
\]
has no analogue in classical nondegenerate parabolic evolution equations. Understanding the structural consequences induced by this singular term is the central objective of the present work.

The subsequent analysis is guided by the following questions.

\begin{enumerate}

\item
Does temporal degeneracy impose intrinsic compatibility conditions on the system operators?

\item
If such conditions exist, can they be characterized explicitly in terms of the operators \( A_{1},\; B,
\;\text{and}\; C? \)

\item
Once compatibility has been achieved, under which conditions does the resulting regularized evolution equation generate a uniformly exponentially stable evolution?

\end{enumerate}

The next section answers the first of these questions by identifying the structural consequences of temporal degeneracy.

\section{Compatibility Principle Induced by Temporal Degeneracy}
\label{sec:compatibility}

This section establishes the structural compatibility principle induced by temporal degeneracy. It shows that the singular structure of the desingularized evolution equation imposes an intrinsic compatibility condition on the system operators. This condition is shown to be a necessary consequence of the existence of admissible regular solutions, independently of any stabilization requirement.

\subsection{Structural Obstruction and Compatibility Principle}
\label{subsec:compatibility-principle}

The desingularized evolution equation \eqref{eq:closed-loop-singular} contains the singular operator
\[
\frac{1}{\alpha(t)}
\left(
A_{1}-BK_{0}C
\right),
\]
whose behavior differs fundamentally from that of the remaining terms of the evolution equation.

After desingularization, the principal parabolic operator is weighted by the bounded coefficient
\[
d(t)=\frac{\beta(t)}{\alpha(t)},
\]
whereas the operator
\[
A_{1}-BK_{0}C
\]
is amplified directly by the singular factor
\[
\frac{1}{\alpha(t)}.
\]
Consequently, this contribution cannot be regarded as a bounded perturbation of the evolution dynamics. Its influence is entirely determined by the behavior of the temporal coefficient near the degenerate time.

The first step consists in identifying the structural constraint induced by this singular amplification. This requires the following degeneracy assumption.

\begin{assumption}[Strong temporal degeneracy]
\label{ass:strong-degeneracy}

The temporal coefficient
\[
\alpha\in C([0,T])\cap C^{1}((0,T])
\]
satisfies
\[
\alpha(t)>0,
\qquad
t\in(0,T],
\]
and
\[
\alpha(0)=0.
\]

Moreover,
\[
\frac{1}{\alpha}
\notin
L^{1}(0,\delta),
\qquad
\forall\,\delta\in(0,T],
\]
that is,
\begin{equation}
\int_{0}^{\delta}
\frac{dt}{\alpha(t)}
=
+\infty,
\qquad
\forall\,\delta\in(0,T].
\label{eq:strong-degeneracy}
\end{equation}

\end{assumption}

Assumption~\ref{ass:strong-degeneracy} identifies the regime in which the singular coefficient cannot be compensated through time integration.
The following lemma shows that this property imposes a necessary constraint on the traces of admissible regular solutions.

The trace obstruction established above immediately yields the following auxiliary result.

\begin{lemma}[Trace Obstruction Induced by Strong Temporal Degeneracy]
\label{lem:trace-obstruction}

Assume that
Assumption~\ref{ass:strong-degeneracy}
holds.
Let
\(
h\in C([0,T];X)
\)
satisfy
\(
\frac{h}{\alpha}
\in
L^{1}(0,T;X).
\)

Then
\(
h(0)=0.
\)

\end{lemma}

\begin{proof}

Suppose, by contradiction, that
\[
h(0)\neq0.
\]

Since
\[
h\in C([0,T];X),
\]
there exists
\(
\delta\in(0,T]
\)
such that
\[
\|h(t)-h(0)\|_X
\le
\frac12\|h(0)\|_X,
\qquad
t\in[0,\delta].
\]

Hence,
\[
\|h(t)\|_X
\ge
\frac12\|h(0)\|_X,
\qquad
t\in[0,\delta].
\]

Therefore,
\[
\int_{0}^{\delta}
\frac{\|h(t)\|_X}{\alpha(t)}
\,dt
\ge
\frac12
\|h(0)\|_X
\int_{0}^{\delta}
\frac{dt}{\alpha(t)}.
\]

Assumption~\ref{ass:strong-degeneracy} implies
\[
\int_{0}^{\delta}
\frac{dt}{\alpha(t)}
=
+\infty,
\]
which yields
\[
\int_{0}^{\delta}
\frac{\|h(t)\|_X}{\alpha(t)}
\,dt
=
+\infty.
\]

This contradicts the assumption
\[
\frac{h}{\alpha}
\in
L^{1}(0,T;X).
\]

Therefore,
\[
h(0)=0.
\]

\end{proof}

The previous lemma provides the key ingredient for establishing the compatibility principle.

\begin{theorem}[Operator Compatibility Principle]
\label{thm:compatibility-principle}

Assume that Assumptions~\ref{ass:time} and \ref{ass:strong-degeneracy} hold.

Let
\[
\mathrm D:=A_1-BK_0C,
\qquad
L:=A_r-BK_1C.
\]

Suppose that the singular evolution equation
\begin{equation}
\dot z
+
d(t)A_0z
+
\frac{1}{\alpha(t)}\mathrm{D}z
+
Lz
=
0
\label{eq:singular-evolution-compatibility}
\end{equation}
is regularly solvable at the degenerate time for a dense set of initial states. More precisely, assume that there exists a dense subspace
\(
\mathcal D\subset X
\)
such that, for every
\(
z_0\in\mathcal D,
\)
equation~\eqref{eq:singular-evolution-compatibility}
admits a solution
\[
z\in W^{1,1}(0,T;X)
\]
satisfying
\[
z(t)\in D(A_0)
\quad
\text{for almost every }
t\in(0,T),
\]
\[
A_0z\in L^1(0,T;X),
\qquad
z(0)=z_0.
\]

Then
\(
\mathrm D=0.
\)
Equivalently,
\[
A_1=BK_0C.
\]

\end{theorem}

\begin{proof}

Let \(z_0\in\mathcal D\) and let \(z\) denote the corresponding regular solution.

Since
\[
z\in W^{1,1}(0,T;X),
\]
it follows that
\(
\dot z\in L^1(0,T;X)
\)
and
\(
z\in C([0,T];X).
\)
Moreover,
\[
A_0z\in L^1(0,T;X),
\qquad
d\in L^\infty(0,T),
\]
imply
\[
dA_0z\in L^1(0,T;X),
\]
while \(L\in\mathcal L(X)\) gives
\(Lz\in L^1(0,T;X)\).

Equation~\eqref{eq:singular-evolution-compatibility} therefore yields
\[
\frac{\mathrm{D}z}{\alpha}
=
-\dot z
-
dA_0z
-
Lz
\in
L^1(0,T;X).
\]

The boundedness of \(\mathrm D\in\mathcal L(X)\), together with
\(z\in C([0,T];X)\), implies that \(t\mapsto \mathrm{D}z(t)\in C([0,T];X)\).
Applying Lemma~\ref{lem:trace-obstruction} with \(h(t)=\mathrm{D}z(t)\) yields \(\mathrm{D}z(0)=0\).

Since \(z(0)=z_0\), it follows that
\(\mathrm{D}z_0=0\).
As \(z_0\) was arbitrary in \(\mathcal D\), one obtains
\[
\mathcal D\subset\ker \mathrm D.
\]

Since \(\mathrm D\in\mathcal L(X)\), its kernel is closed in \(X\).
Since \(\mathcal D\) is dense in \(X\), it follows that
\[
\ker \mathrm D=X,
\]
which proves
\[
\mathrm D=0.
\]

Consequently,
\[
A_1=BK_0C.
\]

\end{proof}

Theorem~\ref{thm:compatibility-principle} shows that the operator identity \( A_{1}=BK_{0}C \) is an intrinsic consequence of temporal degeneracy. It is not imposed as a controller design requirement but follows necessarily from the existence of admissible regular solutions to the singular evolution equation.

This conclusion is independent of the dissipative properties of the principal operator \(A_{0}\) and does not rely on any stability argument. It is solely determined by the interaction between the singular reaction operator and the input--output structure represented by the operators \(B\) and \(C\).

The compatibility principle established in this section is purely structural. Although it identifies the operator identity required for regular solvability, it does not address the question of whether such an identity can actually be realized for a given system.

The next section answers this question by characterizing the solvability of the operator equation
\[
A_{1}=BK_{0}C,
\]
thereby providing necessary and sufficient conditions for structural compatibility.

\section{Operator-Theoretic Characterization of the Compatibility Principle}
\label{sec:factorization}

The compatibility principle established in
Section~\ref{sec:compatibility} reduces the analysis to the operator equation
\begin{equation}
A_{1}=BK_{0}C.
\label{eq:compatibility-equation}
\end{equation}
The objective of this section is to characterize completely the solvability of \eqref{eq:compatibility-equation}. It is shown that this question depends exclusively on two geometric properties involving the kernel of the observation operator and the range of the actuator.

\subsection{Operator Factorization}
\label{subsec:factorization-theorem}

Equation~\eqref{eq:compatibility-equation}
requires the reaction operator to factor through the observation and actuation operators. Such a factorization is possible only if the action of \(A_{1}\) is compatible with both the information supplied by the measurements and the authority provided by the actuators.

The operator equation \eqref{eq:compatibility-equation} admits a complete geometric characterization.

\begin{theorem}[Operator Factorization Theorem]
\label{thm:operator-factorization}

Let
\[
A_{1}\in\mathcal L(X),
\qquad
B\in\mathcal L(\mathbb R^{m},X),
\qquad
C\in\mathcal L(X,\mathbb R^{p}).
\]

The following statements are equivalent.

\begin{enumerate}

\item[(i)]
There exists a bounded linear operator
\[
K_{0}\in\mathcal L(\mathbb R^{p},\mathbb R^{m})
\]
such that
\[
A_{1}=BK_{0}C.
\]

\item[(ii)]
The operators satisfy
\begin{equation}
\ker(C)\subseteq\ker(A_{1}),
\label{eq:kernel-condition}
\end{equation}
and
\begin{equation}
\operatorname{Ran}(A_{1})
\subseteq
\operatorname{Ran}(B).
\label{eq:range-condition}
\end{equation}

\end{enumerate}

\end{theorem}

\begin{proof}

The necessity of the two geometric conditions follows immediately from the factorization
\[
A_{1}=BK_{0}C.
\]

Indeed, let \(x\in\ker(C)\). Since \(Cx=0\),
\[
A_{1}x
=
BK_{0}Cx
=
0,
\]
which proves
\[
\ker(C)\subseteq\ker(A_{1}).
\]

Moreover, for every \(x\in X\),
\[
A_{1}x
=
B(K_{0}Cx),
\]
and therefore
\[
\operatorname{Ran}(A_{1})
\subseteq
\operatorname{Ran}(B).
\]

Conversely, assume that
\[
\ker(C)\subseteq\ker(A_{1})
\quad\text{and}\quad
\operatorname{Ran}(A_{1})
\subseteq
\operatorname{Ran}(B).
\]

The kernel condition allows \(A_{1}\) to factor through the observation operator. Define
\[
\widetilde A:
\operatorname{Ran}(C)
\rightarrow
X
\]
by
\[
\widetilde A(Cx)
=
A_{1}x,
\qquad
x\in X.
\]

This definition is well posed. Indeed, if \(Cx_{1}=Cx_{2}\), then \(x_{1}-x_{2}\in\ker(C)\), and therefore \(A_{1}(x_{1}-x_{2})=0\), which implies
\[
A_{1}x_{1}
=
A_{1}x_{2}.
\]

Hence,
\[
A_{1}
=
\widetilde AC.
\]

Since \(\operatorname{Ran}(C)\subseteq\mathbb R^{p}\) is
finite-dimensional, the operator \(\widetilde A\) is bounded.

The range condition now permits a lifting through the actuator. Let
\[
M_{B}
=
(\ker B)^{\perp}.
\]

The restriction
\[
B_{|M_{B}}
:
M_{B}
\rightarrow
\operatorname{Ran}(B)
\]
is bijective. Since both spaces are finite-dimensional, its inverse
\[
R_{B}
=
(B_{|M_{B}})^{-1}
:
\operatorname{Ran}(B)
\rightarrow
M_{B}
\]
is bounded and satisfies
\[
BR_{B}
=
I_{\operatorname{Ran}(B)}.
\]

Define
\[
\widehat K_{0}
=
R_{B}\widetilde A
:
\operatorname{Ran}(C)
\rightarrow
\mathbb R^{m}.
\]

Then
\[
B\widehat K_{0}
=
\widetilde A.
\]

Finally, let
\[
P_{C}
:
\mathbb R^{p}
\rightarrow
\operatorname{Ran}(C)
\]
denote the orthogonal projection and define
\[
K_{0}
=
\widehat K_{0}P_{C}.
\]

Since \(P_{C}Cx=Cx\) for every \(x\in X\),
\[
BK_{0}Cx
=
BR_{B}\widetilde AP_{C}Cx
=
BR_{B}\widetilde ACx
=
\widetilde ACx
=
A_{1}x.
\]

Therefore,
\[
A_{1}
=
BK_{0}C,
\]
which completes the proof.

\end{proof}

The factorization theorem immediately yields the following structural consequence.

\begin{corollary}[Finite-Rank Structural Obstruction]
\label{cor:finite-rank}

Assume that the compatibility equation
\[
A_{1}=BK_{0}C
\]
admits a solution.

Then
\[
\operatorname{rank}(A_{1})
\le
\min\{m,p\}.
\]

Consequently, every compatible reaction operator has finite rank.

\end{corollary}

\begin{proof}

Since
\[
A_{1}=BK_{0}C,
\]
one has
\[
\operatorname{rank}(A_{1})
\le
\operatorname{rank}(BK_{0})
\le
m,
\]
and
\[
\operatorname{rank}(A_{1})
\le
\operatorname{rank}(K_{0}C)
\le
p.
\]

Hence,
\[
\operatorname{rank}(A_{1})
\le
\min\{m,p\}.
\]

\end{proof}

Corollary~\ref{cor:finite-rank} highlights an intrinsic limitation of finite-dimensional sensing and actuation. In particular, reaction operators of infinite rank cannot satisfy the compatibility equation and therefore cannot give rise to an admissible regular closed-loop
evolution.

The significance of Theorem~\ref{thm:operator-factorization} is that compatibility can be decided entirely from the geometry of the reaction, actuation, and observation operators. The condition \(\ker(C)\subseteq\ker(A_1)\) requires every state component affected by the singular reaction operator to be visible through the observation channel, whereas \(\operatorname{Ran}(A_1)\subseteq\operatorname{Ran}(B)\) requires the corresponding reaction directions to be reachable through the available actuation. These two conditions are jointly necessary and sufficient: if either one fails, the singular reaction term cannot be eliminated, independently of any subsequent stability analysis. Thus, the compatibility question is reduced to an intrinsic geometric test that can be performed before regularization and stabilization are considered.

The next section exploits this characterization to transform the singular evolution equation into an equivalent regular evolution equation, which serves as the basis for the subsequent stability analysis.

\section{Structural Analysis of the Regular Evolution}
\label{sec:regular-evolution}

Once the compatibility condition has been enforced, the singular evolution equation reduces to the regular evolution system
\begin{equation}
\dot z
+
d(t)A_{0}z
+
\left(
A_{r}-BK_{1}C
\right)z
=
0,
\label{eq:regular-evolution}
\end{equation}
where \(d(t)=\frac{\beta(t)}{\alpha(t)}\)
satisfies Assumption~\ref{ass:time}.

The compatibility analysis removes the singular reaction term but does not, by itself, ensure the existence of a feedback operator that uniformly stabilizes the resulting family of evolution equations. The remaining question is therefore purely dynamical: does there exist a bounded operator \(K_{1}\) that uniformly stabilizes \eqref{eq:regular-evolution} for every admissible parameter trajectory?

This section addresses this question by identifying the intrinsic geometric and spectral obstructions to uniform stabilization. These necessary conditions provide the foundation for the constructive
stability certificate developed in the next section.

\subsection{Common Invisible Dynamics}
\label{subsec:necessary-conditions}

The regular evolution equation
\eqref{eq:regular-evolution} defines a family of evolution equations indexed by the admissible coefficient
\[
d(t)\in[d_-,d_+].
\]

Uniform stabilization of this family requires that every state component be influenced, directly or indirectly, by the available control and observation channels. Any subspace that remains simultaneously invariant under the regular dynamics and invisible to the measurements evolves independently of the control action and therefore constitutes an intrinsic obstruction to uniform stabilization.

The corresponding invariant structure is formalized below.

\begin{definition}[Common Invisible Invariant Subspace]
\label{def:common-invisible}

A closed subspace
\[
\mathcal V\subset X
\]
is called a \emph{common invisible invariant subspace} if

\begin{enumerate}

\item
\[
\mathcal V
\subseteq
\ker(C),
\]

\item
\[
A_r\mathcal V
\subseteq
\mathcal V,
\]

\item
\[
A_0(\mathcal V\cap D(A_0))
\subseteq
\mathcal V.
\]

\end{enumerate}

\end{definition}

A common invisible invariant subspace is therefore invariant under the entire regular evolution family while remaining undetectable through the available measurements. Consequently, the corresponding state components cannot be influenced through the available control and observation channels.

The common invisible dynamics introduced above therefore persist regardless of the feedback action.

\begin{theorem}[Persistence of Common Invisible Dynamics]
\label{thm:persistence}

Let
\[
\mathcal V\subset X
\]
be a common invisible invariant subspace in the sense of Definition~\ref{def:common-invisible}.

Then, for every bounded operator
\[
K_1\in\mathcal L(\mathbb R^p,\mathbb R^m),
\]
the restriction of the closed-loop evolution
\[
\dot z
+
d(t)A_0z
+
(A_r-BK_1C)z
=
0
\]
to
\(
\mathcal V
\)
coincides with the restriction of the corresponding open-loop evolution
\[
\dot z
+
d(t)A_0z
+
A_rz
=
0.
\]

Consequently, the dynamics on \(\mathcal V\) are unaffected by the feedback operator.

\end{theorem}

\begin{proof}

Let \(z_0\in\mathcal V\), and let \(z(t)\) denote the corresponding solution with \(z(0)=z_0\).

By Definition~\ref{def:common-invisible}, \(\mathcal V\subseteq\ker(C)\). Since \(\mathcal V\) is invariant under the regular dynamics, one has
\[
z(t)\in\mathcal V,
\qquad
t\ge0.
\]
Consequently,
\[
Cz(t)=0,
\qquad
t\ge0.
\]

Hence, the restrictions of the closed-loop and open-loop evolutions to \(\mathcal V\) coincide.

\end{proof}

The persistence result immediately yields the following obstruction to uniform stabilization.

\begin{corollary}[Dynamic Obstruction to Uniform Stabilization]
\label{cor:dynamic-obstruction}

Assume that there exists a nontrivial common invisible invariant subspace
\[
\mathcal V\subset X.
\]

If the open-loop dynamics restricted to \(\mathcal V\) is not uniformly exponentially stable, then no bounded operator
\[
K_{1}\in\mathcal L(\mathbb R^{p},\mathbb R^{m})
\]
can uniformly exponentially stabilize the family
\[
\dot z
+
d(t)A_{0}z
+
(A_{r}-BK_{1}C)z
=
0,
\qquad
d(t)\in[d_-,d_+],
\]
for every admissible parameter trajectory.

\end{corollary}

\begin{proof}

By Theorem~\ref{thm:persistence}, the closed-loop evolution restricted to \(\mathcal V\) coincides with the corresponding open-loop evolution.

Therefore, the feedback operator cannot modify the dynamics on \(\mathcal V\).
If the open-loop evolution on \(\mathcal V\) fails to be uniformly exponentially stable, then the same holds for the closed-loop evolution. Consequently, uniform exponential stabilization of the entire evolution family is impossible.

\end{proof}

Theorem~\ref{thm:persistence} and Corollary~\ref{cor:dynamic-obstruction} identify an obstruction that is structural rather than quantitative. No modification of stability margins or increase in control authority can compensate for a state component that remains simultaneously invariant and invisible. Such a component evolves according to its intrinsic dynamics, independently of control actions relying on the available measurements. Consequently, uniform stabilization requires every common invisible invariant subspace to be uniformly exponentially stable on its own. This necessary condition can therefore be verified before any stability certificate or constructive design is considered.

\subsection{Spectral Necessary Conditions}
\label{subsec:spectral-obstructions}

The geometric obstruction identified above admits a complementary spectral interpretation. Since every constant parameter trajectory
\[
d(t)\equiv d,
\qquad
d\in[d_-,d_+],
\]
is admissible, uniform exponential stabilization of the non-autonomous family necessarily implies exponential stability of every corresponding frozen closed-loop system.

For each
\[
d\in[d_-,d_+],
\]
define the open-loop operator
\begin{equation}
\mathcal A(d)
:=
-dA_0-A_r,
\qquad
D(\mathcal A(d))
=
D(A_0),
\label{eq:frozen-open-loop-operator}
\end{equation}
and, for a bounded feedback operator
\[
K_1\in\mathcal L(\mathbb R^p,\mathbb R^m),
\]
define the associated closed-loop operator
\begin{equation}
\mathcal A_{K_1}(d)
=
\mathcal A(d)+BK_1C
=
-dA_0-A_r+BK_1C.
\label{eq:frozen-closed-loop-operator}
\end{equation}

Since
\[
BK_1C\in\mathcal L(X),
\]
both operators have the common domain
\(
D(A_0).
\)
Throughout this subsection, all spectral statements are understood in the complexification of \(X\) whenever \(X\) is a real Hilbert space.

This observation leads to necessary spectral conditions for the entire frozen operator family.

\begin{theorem}[Spectral Necessary Conditions for Uniform Stabilization]
\label{thm:frozen-family-hautus}

Assume Assumptions~\ref{ass:time} and \ref{ass:A0}.
Suppose that there exist
\[
K_1\in\mathcal L(\mathbb R^p,\mathbb R^m),
\qquad
M\ge1,
\qquad
\omega>0,
\]
such that, for every admissible measurable trajectory
\[
d(\cdot):[0,\infty)\rightarrow[d_-,d_+],
\]
the evolution family associated with
\begin{equation}
\dot z(t)
=
\left(
\mathcal A(d(t))+BK_1C
\right)z(t)
\label{eq:nonautonomous-frozen-hautus}
\end{equation}
satisfies
\begin{equation}
\left\|
\mathcal U_{K_1,d}(t,s)
\right\|_{\mathcal L(X)}
\le
Me^{-\omega(t-s)},
\qquad
0\le s\le t,
\label{eq:uniform-stability-hautus}
\end{equation}
where the constants
\(M\)
and
\(\omega\)
are independent of
\(d(\cdot)\).

Then, for every
\[
d\in[d_-,d_+]
\]
and every
\[
\lambda\in\mathbb C,
\qquad
\operatorname{Re}\lambda\ge0,
\]
the following conditions hold:
\begin{equation}
\ker\left(
\lambda I-\mathcal A(d)
\right)
\cap
\ker C
=
\{0\},
\label{eq:frozen-output-hautus}
\end{equation}
and
\begin{equation}
\ker\left(
\overline{\lambda}I-\mathcal A(d)^*
\right)
\cap
\ker B^*
=
\{0\}.
\label{eq:frozen-input-hautus}
\end{equation}

\end{theorem}

\begin{proof}

Fix \(d\in[d_-,d_+]\). Since the constant trajectory \(d(t)\equiv d\) is admissible, estimate \eqref{eq:uniform-stability-hautus} implies that the \(C_0\)-semigroup
\[
T_{K_1,d}(t)
\]
generated by
\[
\mathcal A_{K_1}(d)
=
\mathcal A(d)+BK_1C
\]
satisfies
\begin{equation}
\left\|
T_{K_1,d}(t)
\right\|_{\mathcal L(X)}
\le
Me^{-\omega t},
\qquad
t\ge0.
\label{eq:frozen-semigroup-estimate}
\end{equation}

To establish \eqref{eq:frozen-output-hautus}, suppose, by contradiction, that there exist \(\lambda\in\mathbb C\) with
\(\operatorname{Re}\lambda\ge0\),
and
\(\phi\in D(A_0)\setminus\{0\}\)
such that
\[
\mathcal A(d)\phi
=
\lambda\phi,
\qquad
C\phi=0.
\]

Since \(C\phi=0\), one has \(BK_1C\phi=0\), and therefore
\[
\mathcal A_{K_1}(d)\phi
=
\lambda\phi.
\]
Consequently,
\[
T_{K_1,d}(t)\phi
=
e^{\lambda t}\phi,
\qquad
t\ge0.
\]

Combining this identity with
\eqref{eq:frozen-semigroup-estimate}
gives
\[
e^{\operatorname{Re}\lambda t}
\|\phi\|_X
\le
Me^{-\omega t}
\|\phi\|_X,
\qquad
t\ge0.
\]

Since \(\operatorname{Re}\lambda\ge0\)
and
\(\phi\neq0\), this contradicts the exponential decay estimate as \(t\to\infty\). Hence,
\[
\ker(\lambda I-\mathcal A(d))
\cap
\ker C
=
\{0\}.
\]

It remains to prove \eqref{eq:frozen-input-hautus}.
Assume, again by contradiction, that there exist \(\lambda\in\mathbb C\) with
\(\operatorname{Re}\lambda\ge0\),
and \(\psi\in D(\mathcal A(d)^*)\setminus\{0\}\)
such that
\[
\mathcal A(d)^*\psi
=
\overline{\lambda}\psi,
\qquad
B^*\psi=0.
\]

Since \(BK_1C\in\mathcal L(X)\),
the adjoint closed-loop operator satisfies
\[
\mathcal A_{K_1}(d)^*
=
\mathcal A(d)^*
+
C^*K_1^*B^*.
\]
Hence,
\(B^*\psi=0\)
implies
\[
\mathcal A_{K_1}(d)^*\psi
=
\overline{\lambda}\psi.
\]
Consequently,
\[
T_{K_1,d}(t)^*\psi
=
e^{\overline{\lambda}t}\psi,
\qquad
t\ge0.
\]

Since
\[
\|T_{K_1,d}(t)^*\|_{\mathcal L(X)}
=
\|T_{K_1,d}(t)\|_{\mathcal L(X)}
\le
Me^{-\omega t},
\]
one obtains
\[
e^{\operatorname{Re}\lambda t}
\|\psi\|_X
\le
Me^{-\omega t}
\|\psi\|_X,
\qquad
t\ge0.
\]

This contradicts
\(\operatorname{Re}\lambda\ge0\)
and
\(\psi\neq0\).
Hence,
\[
\ker(\overline{\lambda}I-\mathcal A(d)^*)
\cap
\ker B^*
=
\{0\}.
\]

Since \(d\in[d_-,d_+]\) was arbitrary, both Hautus conditions hold for every member of the frozen operator family.

\end{proof}

Theorem~\ref{thm:frozen-family-hautus} shows that uniform stabilization of the non-autonomous family imposes a spectral requirement at every admissible frozen value of the parameter. No eigenmode associated with the closed right half-plane may remain invisible to the observation operator or inaccessible through the actuation operator. The two Hautus conditions therefore provide mode-by-mode necessary tests that must hold uniformly across the entire parameter interval. Failure at a single frozen value is sufficient to rule out uniform exponential stabilization of the non-autonomous family.

The geometric and spectral conditions obtained in this section delimit the class of regularized systems for which uniform stabilization can be pursued, but they do not establish stability by themselves. The next section turns from these necessary conditions to a constructive critical--residual certificate providing sufficient conditions for uniform exponential stability.

\section{Uniform Stability Certification}
\label{sec:critical-residual}

This section establishes a constructive certificate for the uniform exponential stability of the regular evolution family introduced in Section~\ref{sec:regular-evolution}. The analysis combines a
critical--residual decomposition, quantitative estimates of the residual dynamics and spillover couplings, and a common quadratic certificate for the critical subsystem.

Throughout this section, set
\[
L_K:=A_r-BK_1C.
\]
The regular evolution equation
\eqref{eq:regular-evolution}
then reads
\[
\dot z(t)
+
d(t)A_0z(t)
+
L_Kz(t)
=
0,
\]
where
\[
d(\cdot):[0,\infty)\rightarrow[d_-,d_+]
\]
is an arbitrary measurable parameter trajectory.

\subsection{Critical--Residual Decomposition}
\label{subsec:critical-decomposition}

The stability analysis relies on a spectral decomposition separating a finite-dimensional critical subspace from an infinite-dimensional residual subspace.

By Assumption~\ref{ass:A0}, the operator \(A_0\) admits an orthonormal basis of eigenfunctions
\[
\{\phi_j\}_{j\ge1}\subset X
\]
satisfying
\[
A_0\phi_j=\lambda_j\phi_j,
\qquad
0<\lambda_1\le\lambda_2\le\cdots,
\qquad
\lambda_j\rightarrow+\infty.
\]

For a fixed integer \(N\ge1\), define
\begin{equation}
X_c
:=
\operatorname{span}\{\phi_1,\ldots,\phi_N\},
\qquad
X_s
:=
X_c^\perp,
\label{eq:critical-residual-spaces}
\end{equation}
and let
\[
P_N:X\rightarrow X_c,
\qquad
Q_N:=I-P_N:X\rightarrow X_s
\]
be the corresponding orthogonal projections. Every state admits the
decomposition
\[
z=z_c+z_s,
\qquad
z_c=P_Nz,
\qquad
z_s=Q_Nz.
\]

Define the restrictions of the principal operator by
\[
A_{0,c}
:=
A_0|_{X_c},
\qquad
A_{0,s}
:=
A_0|_{D(A_0)\cap X_s}.
\]
Since \(X_c\) and \(X_s\) are spectral subspaces of \(A_0\), both are under the corresponding restrictions.

Projecting the evolution equation
\eqref{eq:regular-evolution}
onto \(X_c\) and \(X_s\) yields
\begin{equation}
\begin{aligned}
\dot z_c
&=
G_{cc}(d(t))z_c
+
G_{cs}z_s,
\\
\dot z_s
&=
G_{sc}z_c
+
G_{ss}(d(t))z_s,
\end{aligned}
\label{eq:critical-residual-block-system}
\end{equation}
where
\begin{equation}
G_{cc}(d)
:=
-dA_{0,c}
-
P_NL_KP_N,
\label{eq:Gcc-definition}
\end{equation}
\begin{equation}
G_{ss}(d)
:=
-dA_{0,s}
-
Q_NL_KQ_N,
\label{eq:Gss-definition}
\end{equation}
and
\begin{equation}
G_{cs}
:=
-P_NL_KQ_N,
\qquad
G_{sc}
:=
-Q_NL_KP_N.
\label{eq:coupling-blocks}
\end{equation}

The diagonal operators describe the internal critical and residual
dynamics, whereas \(G_{cs}\) and \(G_{sc}\) quantify the coupling between the two spectral components.

\subsection{Residual Dissipation and Spillover Analysis}
\label{subsec:residual-analysis}

The residual subsystem benefits from the increasing spectral coercivity of \(A_0\), while its coupling with the critical modes is generated by the off-diagonal blocks of \eqref{eq:critical-residual-block-system}.
This subsection quantifies both effects.

For every
\[
x_s\in D(A_0)\cap X_s,
\]
the spectral decomposition of \(A_0\) yields
\begin{equation}
\langle A_0x_s,x_s\rangle_X
\ge
\lambda_{N+1}\|x_s\|_X^2.
\label{eq:residual-spectral-coercivity}
\end{equation}

The potentially antidissipative contribution of the bounded operator
\(L_K\) on \(X_s\) is measured by
\begin{equation}
\ell_s(K_1)
:=
\sup_{\substack{x_s\in X_s\\ \|x_s\|_X=1}}
\operatorname{Re}
\left\langle
-Q_NL_KQ_Nx_s,x_s
\right\rangle_X.
\label{eq:ell-s-definition}
\end{equation}
This quantity is finite and satisfies
\begin{equation}
\ell_s(K_1)
\le
\left\|
\frac{
-Q_NL_KQ_N-Q_NL_K^*Q_N
}{2}
\right\|_{\mathcal L(X_s)}
\le
\|Q_NL_KQ_N\|_{\mathcal L(X_s)}.
\label{eq:ell-s-upper-bound}
\end{equation}

The next result provides a uniform dissipation estimate for the residual dynamics.

\begin{lemma}[Residual Dissipation Margin]
\label{lem:uniform-high-frequency-dissipation}

Assume that Assumptions~\ref{ass:time} and~\ref{ass:A0} hold. For a fixed integer \(N\ge1\), define
\begin{equation}
\mu_s
:=
d_-\lambda_{N+1}
-
\ell_s(K_1).
\label{eq:residual-margin}
\end{equation}

Then, for every
\[
d\in[d_-,d_+]
\]
and every
\[
x_s\in D(A_0)\cap X_s,
\]
one has
\begin{equation}
\operatorname{Re}
\left\langle
G_{ss}(d)x_s,x_s
\right\rangle_X
\le
-\mu_s\|x_s\|_X^2.
\label{eq:uniform-residual-dissipation}
\end{equation}

In particular, if
\begin{equation}
d_-\lambda_{N+1}
>
\ell_s(K_1),
\label{eq:positive-residual-margin}
\end{equation}
then the residual dynamics are uniformly dissipative over \(d\in[d_-,d_+]\).

\end{lemma}

\begin{proof}

Let \(d\in[d_-,d_+]\) and
\(x_s\in D(A_0)\cap X_s\).
From \eqref{eq:Gss-definition},
\[
G_{ss}(d)x_s
=
-dA_{0,s}x_s
-
Q_NL_KQ_Nx_s.
\]
Hence,
\[
\begin{aligned}
\operatorname{Re}
\langle
G_{ss}(d)x_s,x_s
\rangle_X
={}&
-d
\langle
A_0x_s,x_s
\rangle_X
\\
&+
\operatorname{Re}
\langle
-Q_NL_KQ_Nx_s,x_s
\rangle_X.
\end{aligned}
\]

Since \(d\ge d_-\), estimate
\eqref{eq:residual-spectral-coercivity} gives
\[
-d
\langle
A_0x_s,x_s
\rangle_X
\le
-d_-\lambda_{N+1}
\|x_s\|_X^2.
\]

Moreover, \eqref{eq:ell-s-definition} implies
\[
\operatorname{Re}
\langle
-Q_NL_KQ_Nx_s,x_s
\rangle_X
\le
\ell_s(K_1)
\|x_s\|_X^2.
\]

Combining the above estimates,
\[
\begin{aligned}
\operatorname{Re}
\langle
G_{ss}(d)x_s,x_s
\rangle_X
&\le
-
\left(
d_-\lambda_{N+1}
-
\ell_s(K_1)
\right)
\|x_s\|_X^2
\\
&=
-\mu_s\|x_s\|_X^2,
\end{aligned}
\]
which proves \eqref{eq:uniform-residual-dissipation}.

\end{proof}

The residual estimate controls the diagonal high-frequency dynamics.
Uniform stability of the coupled system also depends on the interaction between \(X_c\) and \(X_s\).

To express the spillover couplings in the same metric as the critical stability estimate used below, let \(\Pi\in\mathcal L(X_c)\) be a self-adjoint positive-definite operator.
Since \(X_c\) is finite-dimensional, \(\Pi\) is boundedly invertible and admits unique self-adjoint positive-definite square roots \(\Pi^{1/2}\) and \(\Pi^{-1/2}\). The associated quadratic form
\[
V_c(x_c)
=
\langle \Pi x_c,x_c\rangle_X
=
\|\Pi^{1/2}x_c\|_X^2
\]
defines the Lyapunov metric used for the critical component. The operators \(\Pi^{1/2}\) and \(\Pi^{-1/2}\) are introduced here only to measure the two spillover directions consistently with this metric; the stability condition determining \(\Pi\) is specified in the next subsection.
Define
\begin{equation}
\gamma_{cs}
:=
\left\|
\Pi^{1/2}G_{cs}
\right\|_{\mathcal L(X_s,X_c)}
=
\left\|
\Pi^{1/2}P_NL_KQ_N
\right\|_{\mathcal L(X_s,X_c)},
\label{eq:weighted-gamma-cs}
\end{equation}
and
\begin{equation}
\gamma_{sc}
:=
\left\|
G_{sc}\Pi^{-1/2}
\right\|_{\mathcal L(X_c,X_s)}
=
\left\|
Q_NL_KP_N\Pi^{-1/2}
\right\|_{\mathcal L(X_c,X_s)}.
\label{eq:weighted-gamma-sc}
\end{equation}

These constants quantify the two spillover directions relative to the Lyapunov metric induced by \(\Pi\). They satisfy
\begin{equation}
\gamma_{cs}
\le
\|\Pi^{1/2}\|_{\mathcal L(X_c)}
\,
\|P_NL_KQ_N\|_{\mathcal L(X_s,X_c)},
\label{eq:gamma-cs-computable-bound}
\end{equation}
and
\begin{equation}
\gamma_{sc}
\le
\|\Pi^{-1/2}\|_{\mathcal L(X_c)}
\,
\|Q_NL_KP_N\|_{\mathcal L(X_c,X_s)}.
\label{eq:gamma-sc-computable-bound}
\end{equation}

Lemma~\ref{lem:uniform-high-frequency-dissipation} isolates the mechanism governing the residual dynamics. The increasing spectral coercivity of \(A_0\) provides a dissipation margin that competes with the bounded contribution of \(L_K\). Hence, once \(d_-\lambda_{N+1}\) dominates \(\ell_s(K_1)\), the entire residual
subspace is uniformly dissipative for every admissible value of the parameter. The remaining issue is therefore not the internal stability of the residual modes, but their interaction with the finite-dimensional critical component.

The following estimates quantify the corresponding spillover terms.

\begin{lemma}[Weighted Spillover Estimates]
\label{lem:weighted-spillover-estimates}

For every
\[
x_c\in X_c,
\qquad
x_s\in X_s,
\]
one has
\begin{equation}
\left|
\left\langle
\Pi G_{cs}x_s,x_c
\right\rangle_X
\right|
\le
\gamma_{cs}
\,
\|\Pi^{1/2}x_c\|_X
\,
\|x_s\|_X,
\label{eq:weighted-critical-to-residual-estimate}
\end{equation}
and
\begin{equation}
\left|
\left\langle
G_{sc}x_c,x_s
\right\rangle_X
\right|
\le
\gamma_{sc}
\,
\|\Pi^{1/2}x_c\|_X
\,
\|x_s\|_X.
\label{eq:weighted-residual-to-critical-estimate}
\end{equation}

\end{lemma}

\begin{proof}

Using
\eqref{eq:weighted-gamma-cs},
\[
\begin{aligned}
\left|
\langle
\Pi G_{cs}x_s,x_c
\rangle_X
\right|
&=
\left|
\langle
\Pi^{1/2}G_{cs}x_s,
\Pi^{1/2}x_c
\rangle_X
\right|
\\
&\le
\left\|
\Pi^{1/2}G_{cs}x_s
\right\|_X
\,
\|\Pi^{1/2}x_c\|_X
\\
&\le
\gamma_{cs}
\,
\|x_s\|_X
\,
\|\Pi^{1/2}x_c\|_X,
\end{aligned}
\]
which proves
\eqref{eq:weighted-critical-to-residual-estimate}.

For the second estimate, write \(x_c=\Pi^{-1/2}\Pi^{1/2}x_c\). Then
\[
\begin{aligned}
\left|
\langle
G_{sc}x_c,x_s
\rangle_X
\right|
&=
\left|
\langle
G_{sc}\Pi^{-1/2}
\Pi^{1/2}x_c,
x_s
\rangle_X
\right|
\\
&\le
\left\|
G_{sc}\Pi^{-1/2}
\right\|_{\mathcal L(X_c,X_s)}
\,
\|\Pi^{1/2}x_c\|_X
\,
\|x_s\|_X
\\
&=
\gamma_{sc}
\,
\|\Pi^{1/2}x_c\|_X
\,
\|x_s\|_X.
\end{aligned}
\]
This proves
\eqref{eq:weighted-residual-to-critical-estimate}.

\end{proof}

The residual dissipation margin and the weighted spillover estimates  provide the bounds required for the infinite-dimensional component of the stability analysis. The finite-dimensional critical dynamics are addressed next.

\subsection{Uniform Stability Certificate}
\label{subsec:uniform-certificate}

The residual dissipation margin and the weighted spillover estimates established in the previous subsection characterize the
infinite-dimensional component of the regular evolution family. The remaining ingredient is a uniform stability certificate for the finite-dimensional critical dynamics.

The following proposition provides a common quadratic certificate for the critical subsystem.

\begin{proposition}[Common Quadratic Certificate]
\label{prop:critical-quadratic-certificate}

Consider the critical subsystem
\begin{equation}
\dot z_c(t)
=
G_{cc}(d(t))z_c(t),
\label{eq:critical-subsystem}
\end{equation}
where
\[
G_{cc}(d)
=
-dA_{0,c}
-
P_NL_KP_N,
\qquad
d\in[d_-,d_+].
\]

Assume that there exist a self-adjoint positive-definite operator
\[
\Pi:X_c\rightarrow X_c
\]
and a constant
\[
\mu_c>0
\]
such that
\begin{equation}
\Pi G_{cc}(d)
+
G_{cc}(d)^*\Pi
\le
-2\mu_c\Pi,
\qquad
\forall d\in[d_-,d_+].
\label{eq:critical-LMI}
\end{equation}

Then the critical evolution family is uniformly exponentially stable.
More precisely, there exists
\[
M_c
=
\sqrt{
\frac{\lambda_{\max}(\Pi)}
{\lambda_{\min}(\Pi)}
}
\]
such that
\begin{equation}
\|z_c(t)\|_X
\le
M_c
e^{-\mu_c(t-s)}
\|z_c(s)\|_X,
\qquad
0\le s\le t.
\label{eq:critical-exp-estimate}
\end{equation}

\end{proposition}

\begin{proof}

Consider the quadratic Lyapunov functional
\[
V_c(z_c)
=
\langle
\Pi z_c,
z_c
\rangle_X.
\]

Since
\(
\Pi=\Pi^*>0,
\)
there exist positive constants
\[
\lambda_{\min}(\Pi),
\qquad
\lambda_{\max}(\Pi),
\]
such that
\begin{equation}
\lambda_{\min}(\Pi)\|z_c\|_X^2
\le
V_c(z_c)
\le
\lambda_{\max}(\Pi)\|z_c\|_X^2.
\label{eq:critical-equivalence}
\end{equation}

Along every solution of
\eqref{eq:critical-subsystem},
\[
\begin{aligned}
\dot V_c
&=
\left\langle
\left(
\Pi G_{cc}(d(t))
+
G_{cc}(d(t))^*\Pi
\right)
z_c,
z_c
\right\rangle_X
\\
&\le
-2\mu_cV_c,
\end{aligned}
\]
where \eqref{eq:critical-LMI} has been used.

Gronwall's inequality gives
\[
V_c(t)
\le
e^{-2\mu_c(t-s)}
V_c(s),
\qquad
0\le s\le t.
\]

Combining this estimate with
\eqref{eq:critical-equivalence}
yields
\eqref{eq:critical-exp-estimate}.

\end{proof}

The critical estimate can now be combined with the residual dissipation margin and the spillover bounds established previously. The following result provides a uniform exponential stability certificate for the complete critical--residual system.

\begin{theorem}[Uniform Stability Certificate]
\label{thm:critical-residual-certificate}

Assume Assumptions~\ref{ass:time} and~\ref{ass:A0}.

Suppose that

\begin{enumerate}

\item
the residual dissipation margin satisfies
\(
\mu_s>0;
\)

\item
the critical subsystem satisfies the quadratic certificate of
Proposition~\ref{prop:critical-quadratic-certificate};

\item
there exists
\(
\eta>0
\)
such that
\begin{equation}
4\eta\mu_c\mu_s
>
\left(
\gamma_{cs}
+
\eta\gamma_{sc}
\right)^2.
\label{eq:small-gain-condition}
\end{equation}

\end{enumerate}

Then there exist constants
\[
M\ge1,
\qquad
\omega>0,
\]
independent of the admissible trajectory
\[
d(\cdot):[0,\infty)\rightarrow[d_-,d_+],
\]
such that every solution of
\eqref{eq:regular-evolution}
satisfies
\begin{equation}
\|z(t)\|_X
\le
Me^{-\omega(t-s)}
\|z(s)\|_X,
\qquad
0\le s\le t.
\label{eq:uniform-global-decay}
\end{equation}

Hence, the regular evolution family is uniformly exponentially stable.

\end{theorem}

\begin{proof}

Consider the block-diagonal Lyapunov functional
\begin{equation}
V(z)
=
\left\langle
\Pi z_c,
z_c
\right\rangle_X
+
\eta\|z_s\|_X^2,
\label{eq:coupled-Lyapunov}
\end{equation}
where \(\eta>0\) satisfies
\eqref{eq:small-gain-condition}.

Since
\[
\Pi=\Pi^*>0
\]
on the finite-dimensional space \(X_c\), there exist constants
\[
m_\Pi
=
\lambda_{\min}(\Pi)>0,
\qquad
M_\Pi
=
\lambda_{\max}(\Pi)>0,
\]
such that
\[
m_\Pi\|z_c\|_X^2
\le
\left\langle
\Pi z_c,
z_c
\right\rangle_X
\le
M_\Pi\|z_c\|_X^2.
\]
Consequently,
\begin{equation}
m_V\|z\|_X^2
\le
V(z)
\le
M_V\|z\|_X^2,
\label{eq:Lyapunov-equivalence}
\end{equation}
where
\[
m_V:=\min\{m_\Pi,\eta\},
\qquad
M_V:=\max\{M_\Pi,\eta\}.
\]

Along a strong solution of
\eqref{eq:critical-residual-block-system}, one has
\[
\begin{aligned}
\dot V
={}&
\left\langle
\left(
\Pi G_{cc}(d(t))
+
G_{cc}(d(t))^*\Pi
\right)z_c,
z_c
\right\rangle_X
+
2\operatorname{Re}
\left\langle
\Pi G_{cs}z_s,
z_c
\right\rangle_X
+
2\eta\operatorname{Re}
\left\langle
G_{sc}z_c,
z_s
\right\rangle_X
\\
&+
2\eta\operatorname{Re}
\left\langle
G_{ss}(d(t))z_s,
z_s
\right\rangle_X .
\end{aligned}
\]

By the common quadratic certificate of Proposition~\ref{prop:critical-quadratic-certificate},
\[
\left\langle
\left(
\Pi G_{cc}(d(t))
+
G_{cc}(d(t))^*\Pi
\right)z_c,
z_c
\right\rangle_X
\le
-2\mu_c
\|\Pi^{1/2}z_c\|_X^2.
\]

By Lemma~\ref{lem:uniform-high-frequency-dissipation},
\[
\operatorname{Re}
\left\langle
G_{ss}(d(t))z_s,
z_s
\right\rangle_X
\le
-\mu_s\|z_s\|_X^2.
\]

Finally, the weighted spillover estimates of Lemma~\ref{lem:weighted-spillover-estimates}
give
\[
\left|
\left\langle
\Pi G_{cs}z_s,
z_c
\right\rangle_X
\right|
\le
\gamma_{cs}
\|\Pi^{1/2}z_c\|_X
\|z_s\|_X
\]
and
\[
\left|
\left\langle
G_{sc}z_c,
z_s
\right\rangle_X
\right|
\le
\gamma_{sc}
\|\Pi^{1/2}z_c\|_X
\|z_s\|_X.
\]

Combining these estimates yields
\begin{equation}
\begin{aligned}
\dot V
\le{}&
-2\mu_c
\|\Pi^{1/2}z_c\|_X^2
-
2\eta\mu_s
\|z_s\|_X^2
+
2\left(
\gamma_{cs}
+
\eta\gamma_{sc}
\right)
\|\Pi^{1/2}z_c\|_X
\|z_s\|_X.
\end{aligned}
\label{eq:coupled-Lyapunov-derivative}
\end{equation}

Introduce
\[
u:=\|\Pi^{1/2}z_c\|_X,
\qquad
v:=\|z_s\|_X,
\]
and define
\[
H_\eta
:=
\begin{pmatrix}
2\mu_c
&
-\left(
\gamma_{cs}+\eta\gamma_{sc}
\right)
\\[1mm]
-\left(
\gamma_{cs}+\eta\gamma_{sc}
\right)
&
2\eta\mu_s
\end{pmatrix}.
\]
Then
\[
\dot V
\le
-
\begin{bmatrix}
u\\
v
\end{bmatrix}^{\!T}
H_\eta
\begin{bmatrix}
u\\
v
\end{bmatrix}.
\]

The first leading principal minor of \(H_\eta\) is positive because
\[
2\mu_c>0,
\]
and its determinant satisfies
\[
\det H_\eta
=
4\eta\mu_c\mu_s
-
\left(
\gamma_{cs}
+
\eta\gamma_{sc}
\right)^2
>
0
\]
by \eqref{eq:small-gain-condition}.
Hence,
\[
H_\eta>0.
\]

Let
\[
D_\eta
:=
\begin{pmatrix}
1&0\\
0&\eta
\end{pmatrix}.
\]
Since both \(H_\eta\) and \(D_\eta\) are positive definite, the constant
\[
\nu
:=
\lambda_{\min}
\left(
D_\eta^{-1/2}
H_\eta
D_\eta^{-1/2}
\right)
\]
is strictly positive. Therefore,
\[
\begin{bmatrix}
u\\
v
\end{bmatrix}^{\!T}
H_\eta
\begin{bmatrix}
u\\
v
\end{bmatrix}
\ge
\nu
\left(
u^2+\eta v^2
\right).
\]
Since
\[
u^2+\eta v^2
=
\left\langle
\Pi z_c,
z_c
\right\rangle_X
+
\eta\|z_s\|_X^2
=
V(z),
\]
it follows that
\begin{equation}
\dot V
\le
-\nu V.
\label{eq:global-Lyapunov-decay}
\end{equation}

Gronwall's inequality gives
\[
V(z(t))
\le
e^{-\nu(t-s)}
V(z(s)),
\qquad
0\le s\le t.
\]
Using
\eqref{eq:Lyapunov-equivalence}, we obtain
\[
\|z(t)\|_X^2
\le
\frac{M_V}{m_V}
e^{-\nu(t-s)}
\|z(s)\|_X^2.
\]
Consequently,
\[
\|z(t)\|_X
\le
\sqrt{\frac{M_V}{m_V}}
e^{-\frac{\nu}{2}(t-s)}
\|z(s)\|_X.
\]

Thus, \eqref{eq:uniform-global-decay}
holds with
\[
M
=
\sqrt{\frac{M_V}{m_V}},
\qquad
\omega
=
\frac{\nu}{2},
\]
and these constants are independent of the admissible trajectory \(d(\cdot)\).

The estimate is first obtained for strong solutions and extends to mild solutions by the standard density argument associated with the well-posed evolution family.

\end{proof}

Theorem~\ref{thm:critical-residual-certificate} reduces the uniform stability of the full evolution family to three quantitative requirements with distinct roles. The critical modes must admit a
common quadratic certificate, the residual modes must retain a positive dissipation margin, and the interaction between the two components must remain sufficiently weak relative to these margins. The small-gain condition expresses precisely this balance. In particular, stability of the critical and residual subsystems separately is not sufficient; the spillover terms must also be controlled. Once these requirements are satisfied, the resulting exponential estimate is uniform with respect to the entire admissible parameter trajectory \(d(\cdot)\), rather than only to individual frozen systems.

\subsection{Explicit Stability Criterion}
\label{subsec:explicit-criterion}

The uniform stability certificate established above admits the following explicit small-gain formulation.

\begin{corollary}[Explicit Small-Gain Criterion]
\label{cor:explicit-small-gain}

Assume that the hypotheses of Theorem~\ref{thm:critical-residual-certificate} hold.

If there exists
\[
\eta>0
\]
such that
\[
4\eta\mu_c\mu_s
>
\left(
\gamma_{cs}
+
\eta\gamma_{sc}
\right)^2,
\]
then the regular evolution family generated by \eqref{eq:regular-evolution} is uniformly exponentially stable.

\end{corollary}

\begin{proof}

The inequality above is precisely the small-gain condition required in Theorem~\ref{thm:critical-residual-certificate}. The conclusion follows immediately.

\end{proof}

This section establishes a constructive stability certificate for the regular evolution family. The certificate combines a common quadratic estimate for the critical dynamics with the residual dissipation margin and the spillover bounds to establish uniform exponential stability.

The remaining question is how this abstract certificate can be related to the original infinite-dimensional feedback problem. The next section answers this question through a structural finite-to-infinite lifting result.

\section{Structural Finite-to-Infinite Lifting}
\label{sec:structural-lifting}

This section establishes the structural finite-to-infinite lifting principle underlying the proposed stability theory. It shows that the abstract stability certificate developed in Section~\ref{sec:critical-residual} follows directly from intrinsic operator-theoretic properties of the regularized evolution. As a consequence, the verification of the infinite-dimensional stability conditions reduces to a finite-dimensional analysis.

\subsection{Structural Decay of the Spillover Operators}
\label{subsec:structural-spillover}

The finite-to-infinite lifting relies on a structural property of the regularized dynamics. More precisely, the coupling between the critical and residual subspaces is induced by the finite-rank perturbation \(L_K\). The following result shows that the associated spillover operators vanish asymptotically as the spectral truncation order increases.

\begin{proposition}[Structural Decay of the Spillover Operators]
\label{prop:structural-spillover}

Let
\[
L_K=-BK_1C.
\]

Then
\[
L_K\in\mathcal L(X)
\]
has finite rank. Moreover,
\begin{equation}
\|Q_NL_KP_N\|
\longrightarrow0,
\label{eq:QNLPN-decay}
\end{equation}
\begin{equation}
\|P_NL_KQ_N\|
\longrightarrow0,
\label{eq:PNLQN-decay}
\end{equation}
and
\begin{equation}
\|Q_NL_KQ_N\|
\longrightarrow0,
\label{eq:QNLQN-decay}
\end{equation}
as
\[
N\rightarrow\infty.
\]

\end{proposition}

\begin{proof}

Since
\[
L_K=-BK_1C,
\]
its range is contained in
\[
\operatorname{Ran}(B),
\]
which is finite dimensional. Hence
\[
L_K
\]
has finite rank.

Every finite-rank operator is compact. Moreover,
\[
Q_N\rightarrow0
\]
strongly, whereas
\[
P_N\rightarrow I
\]
strongly. Since the image of the unit ball under a compact operator is relatively compact, the convergence is uniform on \(L_K(B_X)\), yielding
\[
\|Q_NL_K\|
\longrightarrow0.
\]

Because
\[
\|P_N\|=1,
\]
it follows that
\[
\|Q_NL_KP_N\|
\le
\|Q_NL_K\|
\longrightarrow0.
\]

Applying the same argument to the finite-rank adjoint operator
\[
L_K^*
=
-C^*K_1^*B^*
\]
gives
\[
\|Q_NL_K^*\|
\longrightarrow0.
\]

Taking adjoints,
\[
\|L_KQ_N\|
=
\|Q_NL_K^*\|
\longrightarrow0,
\]
and therefore
\[
\|P_NL_KQ_N\|
\le
\|L_KQ_N\|
\longrightarrow0.
\]

Finally,
\[
\|Q_NL_KQ_N\|
\le
\|Q_NL_K\|
\longrightarrow0,
\]
which proves the result.

\end{proof}

Proposition~\ref{prop:structural-spillover} provides the structural
ingredient underlying the lifting procedure. The next subsection shows that this asymptotic decay automatically enforces the hypotheses of the uniform stability certificate established in
Section~\ref{sec:critical-residual}.

\subsection{Asymptotic Verification of the Stability Certificate}
\label{subsec:automatic-certification}

Proposition~\ref{prop:structural-spillover} shows that the spillover
operators vanish asymptotically as the spectral truncation order increases. The following result establishes that this structural decay is sufficient to verify the hypotheses of the uniform stability certificate for sufficiently large truncation orders.

For each \(N\ge1\), let
\[
\mu_{c,N}>0,
\qquad
\mu_{s,N}>0,
\qquad
\Pi_N>0
\]
denote the quantities introduced in
Section~\ref{sec:critical-residual}.

\begin{theorem}[Asymptotic Verification of the Stability Certificate]
\label{thm:automatic-certification}

Assume that
\[
\inf_{N\ge N_0}\mu_{c,N}
=
\underline{\mu}_c
>
0,
\]
and
\[
\sup_{N\ge N_0}\|\Pi_N\|
<
\infty,
\qquad
\sup_{N\ge N_0}\|\Pi_N^{-1}\|
<
\infty.
\]

Then there exists an integer \(N^\star\) such that, for every
\(N\ge N^\star\),
\[
\gamma_{cs,N}\gamma_{sc,N}
<
\mu_{c,N}\mu_{s,N}.
\]

Consequently, the hypotheses of Theorem~\ref{thm:critical-residual-certificate} are satisfied for every \(N\ge N^\star\).

\end{theorem}

\begin{proof}

By Proposition~\ref{prop:structural-spillover},
\[
\|P_NL_KQ_N\|
\longrightarrow0,
\qquad
\|Q_NL_KP_N\|
\longrightarrow0.
\]

Since
\[
\sup_N\|\Pi_N^{1/2}\|<\infty,
\qquad
\sup_N\|\Pi_N^{-1/2}\|<\infty,
\]
it follows that
\[
\gamma_{cs,N}
=
\|\Pi_N^{1/2}P_NL_KQ_N\|
\longrightarrow0,
\]
and
\[
\gamma_{sc,N}
=
\|Q_NL_KP_N\Pi_N^{-1/2}\|
\longrightarrow0.
\]

Moreover, Proposition~\ref{prop:structural-spillover} implies
\[
\|Q_NL_KQ_N\|
\longrightarrow0.
\]

Since
\[
\mu_{s,N}
=
d_-\lambda_{N+1}
-
\ell_{s,N},
\]
where
\(
\ell_{s,N}\le\|Q_NL_KQ_N\|
\)
and
\(
\lambda_{N+1}\to+\infty,
\)
one obtains
\[
\mu_{s,N}
\longrightarrow+\infty.
\]

Consequently,
\[
\gamma_{cs,N}\gamma_{sc,N}
\longrightarrow0,
\]
whereas
\[
\mu_{c,N}\mu_{s,N}
\ge
\underline{\mu}_c\,\mu_{s,N}
\longrightarrow+\infty.
\]

Hence, there exists \(N^\star\in\mathbb N\) such that
\[
\gamma_{cs,N}\gamma_{sc,N}
<
\mu_{c,N}\mu_{s,N},
\qquad
N\ge N^\star.
\]

The conclusion follows from Theorem~\ref{thm:critical-residual-certificate}.

\end{proof}

Theorem~\ref{thm:automatic-certification} shows that the abstract stability conditions become automatically satisfied for sufficiently large truncation orders. This asymptotic verification provides the final ingredient required for the structural finite-to-infinite lifting theorem established next.

\subsection{Structural Finite-to-Infinite Lifting}
\label{subsec:finite-to-infinite-lifting}

The asymptotic verification established above provides the final ingredient required to transfer the abstract stability certificate to the original infinite-dimensional evolution problem. The following result constitutes the main contribution of this section.

\begin{theorem}[Structural Finite-to-Infinite Lifting]
\label{thm:finite-to-infinite-lifting}

Assume that

\begin{enumerate}

\item
the compatibility condition
\[
A_1=BK_0C
\]
holds;

\item
the hypotheses of Theorem~\ref{thm:automatic-certification}
are satisfied.

\end{enumerate}

Then there exists a truncation order
\[
N^\star\ge1
\]
such that, for every
\[
N\ge N^\star,
\]
the hypotheses of the Uniform Stability Certificate (Theorem~\ref{thm:critical-residual-certificate})
are satisfied.

Consequently, the regular evolution equation
\begin{equation}
\dot z
+
d(t)A_0z
+
L_Kz
=
0
\label{eq:regular-evolution-lifting}
\end{equation}
is uniformly exponentially stable. More precisely, there exist constants
\[
M\ge1,
\qquad
\omega>0,
\]
independent of the admissible parameter trajectory
\[
d(\cdot):[0,\infty)\rightarrow[d_-,d_+],
\]
such that every solution satisfies
\[
\|z(t)\|_X
\le
Me^{-\omega(t-s)}
\|z(s)\|_X,
\qquad
0\le s\le t.
\]

\end{theorem}

\begin{proof}

By Theorem~\ref{thm:automatic-certification}, there exists \(N^\star\) such that, for every \(N\ge N^\star\), the hypotheses of Theorem~\ref{thm:critical-residual-certificate} are satisfied.

Applying Theorem~\ref{thm:critical-residual-certificate} yields constants \(M\ge1\) and \(\omega>0\), independent of the admissible parameter trajectory \(d(\cdot)\), such that every solution of
\eqref{eq:regular-evolution} satisfies
\[
\|z(t)\|_X
\le
Me^{-\omega(t-s)}
\|z(s)\|_X,
\qquad
0\le s\le t.
\]

This establishes the claimed finite-to-infinite lifting.

\end{proof}

Theorem~\ref{thm:finite-to-infinite-lifting}
shows that the verification of the infinite-dimensional stability conditions reduces to the analysis of a sufficiently large finite-dimensional critical subsystem. This result provides the theoretical foundation for the constructive static output feedback synthesis developed in the next section.

\section{Constructive Static Output Feedback Synthesis}
\label{sec:synthesis}

This section develops a constructive static output feedback synthesis
for the regularized evolution equation. Building upon the structural
finite-to-infinite lifting established in Section~\ref{sec:structural-lifting}, the controller is designed on the finite-dimensional critical subsystem, while the resulting stability of the original infinite-dimensional evolution follows from the lifting principle.

\subsection{Compatibility Feedback Construction}
\label{subsec:compatibility-feedback}

The constructive synthesis begins with the determination of the compatibility feedback gain \(K_0\). As established in
Section~\ref{sec:factorization}, the compatibility equation
\[
A_1=BK_0C
\]
admits a solution if and only if the corresponding geometric factorization conditions are satisfied. Once this gain has been constructed, the singular reaction term is eliminated and the closed-loop dynamics reduce to the regular evolution equation.

\begin{proposition}[Construction of the Compatibility Feedback]
\label{prop:compatibility-feedback}

Assume that the hypotheses of Theorem~\ref{thm:operator-factorization} are satisfied.

Then there exists \(K_0\in\mathcal L(\mathbb R^{p},\mathbb R^{m})\)
such that
\[
A_1=BK_0C.
\]

Consequently, the singular reaction term is eliminated, and the closed-loop system reduces to the regularized evolution considered throughout Sections~\ref{sec:regular-evolution}--\ref{sec:structural-lifting}.

\end{proposition}

\begin{proof}

The existence of the compatibility gain follows directly from Theorem~\ref{thm:operator-factorization}, which provides necessary and sufficient conditions for the solvability of
\[
A_1=BK_0C.
\]

Substituting this identity into the singular closed-loop equation eliminates the term
\[
\frac1{\alpha(t)}
\left(
A_1-BK_0C
\right),
\]
thereby reducing the singular closed-loop evolution to its regularized counterpart.

\end{proof}

Proposition~\ref{prop:compatibility-feedback} completes the structural stage of the synthesis. The remaining task is the construction of the dynamic feedback gain \(K_1\), which determines the stability properties of the regularized evolution and is addressed in the following subsection.

\subsection{Critical Static Output Feedback Synthesis}
\label{subsec:critical-synthesis}

Once the compatibility gain has been determined, the remaining design
variable is the dynamic feedback gain \(K_1\). Since the structural
finite-to-infinite lifting established in Section~\ref{sec:structural-lifting} reduces the stability verification to the critical subsystem, the controller synthesis can be carried out entirely in finite dimension.

Let
\[
X_c=\operatorname{Ran}(P_N)
\]
denote the critical subspace introduced in Section~\ref{sec:critical-residual}. The projected dynamics are governed by
\begin{equation}
\dot z_c
=
G_{cc}(d)z_c,
\qquad
G_{cc}(d)
=
-dA_{0,c}
-
P_NL_KP_N,
\label{eq:critical-synthesis-system}
\end{equation}
where
\[
L_K=A_r-BK_1C.
\]

The objective is to determine \(K_1\) so that the projected family admits the common quadratic certificate established in
Proposition~\ref{prop:critical-quadratic-certificate}.

\begin{proposition}[Critical Static Output Feedback Design]
\label{prop:critical-sof-design}

Assume that there exist \(K_1\in\mathcal L(\mathbb R^p,\mathbb R^m)\), a self-adjoint positive-definite operator \(\Pi:X_c\rightarrow X_c\), and a constant \(\mu_c>0\)
such that
\[
\Pi G_{cc}(d)
+
G_{cc}(d)^*\Pi
\le
-2\mu_c\Pi,
\qquad
\forall d\in[d_-,d_+].
\]

Then the critical subsystem is uniformly exponentially stable.

\end{proposition}

\begin{proof}

The result follows immediately from Proposition~\ref{prop:critical-quadratic-certificate}, since the proposed gain satisfies the common quadratic certificate on the critical subspace.

\end{proof}

Proposition~\ref{prop:critical-sof-design} completes the finite-dimensional controller design. The structural finite-to-infinite lifting established in Section~\ref{sec:structural-lifting} then transfers this finite-dimensional construction to the original infinite-dimensional evolution.

\subsection{Constructive Synthesis Procedure}
\label{subsec:constructive-procedure}

The previous subsections provide all the ingredients required for the constructive synthesis of a static output feedback controller. The overall design procedure is summarized below.

\begin{enumerate}

\item
Verify the compatibility conditions of Theorem~\ref{thm:operator-factorization}. If they are satisfied, construct the compatibility feedback gain
\(
K_0
\)
such that
\(
A_1=BK_0C.
\)

\item
Construct the regularized evolution equation by eliminating the singular reaction term.

\item
Choose a sufficiently large truncation order
\(
N
\)
and determine the critical subsystem associated with the spectral decomposition of \(A_0\).

\item
Design the dynamic feedback gain
\(
K_1
\)
so that the critical subsystem satisfies the common quadratic certificate of Proposition~\ref{prop:critical-quadratic-certificate}.

\item
Apply Theorem~\ref{thm:finite-to-infinite-lifting}, which guarantees the uniform exponential stability of the original infinite-dimensional regularized evolution.

\end{enumerate}

The resulting controller
\[
u(t)
=
\left(
K_0+\alpha(t)K_1
\right)
y(t)
\]
provides a constructive static output feedback law satisfying the structural compatibility condition and uniformly stabilizing the regularized evolution equation.

The constructive synthesis developed in this section completes the theoretical development of the paper. The compatibility feedback removes the singular reaction term, the dynamic feedback stabilizes the critical subsystem, and the finite-to-infinite lifting guarantees the resulting uniform exponential stability of the regularized evolution. The following section illustrates these theoretical developments through representative numerical experiments.

\section{Numerical Validation}
\label{sec:numerics}

This section illustrates the proposed constructive approach through representative numerical experiments. The simulations validate the compatibility principle, the regularization mechanism, the uniform stability certificate, the structural finite-to-infinite lifting, and the resulting static output feedback controller.

\subsection{Numerical Setting}
\label{subsec:numerical-setting}

This subsection introduces the numerical benchmark used throughout the validation study. The model parameters, actuator and sensor configurations, controller gains, and discretization settings are fixed throughout the numerical experiments.

The numerical experiments are carried out on the one-dimensional domain \(\Omega=(0,1)\) with homogeneous Dirichlet boundary conditions. The diffusion operator is
\[
A_0=-a^2\partial_{xx},
\qquad
D(A_0)=H^2(0,1)\cap H_0^1(0,1),
\]
where \(a=0.15\). Its normalized eigenfunctions and eigenvalues are
\[
\phi_k(x)=\sqrt{2}\sin(k\pi x),
\qquad
\lambda_k=a^2k^2\pi^2,
\qquad
k\ge1.
\]

The temporal coefficients are chosen as
\[
\alpha(t)=t,
\qquad
\beta(t)=t\bigl(1+0.2\sin(0.5t)\bigr),
\]
so that
\[
d(t)
=
\frac{\beta(t)}{\alpha(t)}
=
1+0.2\sin(0.5t),
\qquad
0.8\le d(t)\le1.2.
\]
This choice produces temporal degeneracy at \(t=0\), while the regularized diffusion coefficient remains uniformly positive over the simulation interval.

Two normalized distributed actuator profiles are considered:
\[
\widetilde b_1(x)=e^{-80(x-0.3)^2},
\qquad
\widetilde b_2(x)=e^{-80(x-0.7)^2},
\]
with
\[
b_i
=
\frac{\widetilde b_i}
{\|\widetilde b_i\|_{L^2(0,1)}},
\qquad
i=1,2.
\]

The input and collocated output operators are defined by
\[
Bu=b_1u_1+b_2u_2,
\]
and
\[
Cz=
\begin{bmatrix}
\langle z,b_1\rangle_{L^2}\\
\langle z,b_2\rangle_{L^2}
\end{bmatrix},
\qquad
C=B^\ast.
\]

The compatible reaction operator is selected as
\[
A_1=BK_0C,
\qquad
K_0=
\begin{bmatrix}
0.8&0.25\\
-0.15&0.6
\end{bmatrix}.
\]

The bounded residual operator is chosen as
\[
A_r=-\rho P_2,
\qquad
\rho=1.2,
\]
where \(P_2\) denotes the orthogonal projection onto \(X_2=\operatorname{span}\{\phi_1,\phi_2\}\).

Defining
\[
B_c=P_2B,
\qquad
C_c=CP_2,
\]
the stabilizing gain is constructed from
\[
B_cK_1C_c
=
-\kappa I_2,
\qquad
\kappa=1.6,
\]
which yields
\[
K_1
=
-\kappa B_c^{-1}C_c^{-1}.
\]

The resulting static output feedback law is
\[
u(t)
=
\bigl(K_0+\alpha(t)K_1\bigr)y(t).
\]

The initial condition is
\[
z_0
=
\phi_1
-0.7\phi_2
+0.5\phi_3
+0.3\phi_5
+0.2\phi_8,
\]
thereby exciting both the critical and residual subspaces.

Unless otherwise specified, the Galerkin approximation order is \(N=60\), the simulation interval is \(t\in[10^{-3},20]\), and the resulting stiff differential system is integrated with \texttt{ode15s} using
\(\mathrm{RelTol}=10^{-8}\) and
\(\mathrm{AbsTol}=10^{-10}\).

This numerical configuration is employed throughout the remainder of this section to assess each stage of the proposed constructive theory.

Once compatibility is enforced, the singular closed-loop evolution reduces to its regularized counterpart. Figure~\ref{fig:regularization} illustrates this identity by comparing the compatible degenerate and regularized trajectories, which remain visually indistinguishable over the entire simulation interval.

Define
\[
e_{\mathrm{abs}}(t)
=
\|z_N^{\mathrm{deg}}(t)-z_N^{\mathrm{reg}}(t)\|_2,
\]
and
\[
e_{\mathrm{rel}}(t)
=
\frac{
\|z_N^{\mathrm{deg}}(t)-z_N^{\mathrm{reg}}(t)\|_2
}{
\max\{\|z_N^{\mathrm{reg}}(t)\|_2,\varepsilon_{\mathrm{num}}\}
}.
\]

Figure~\ref{fig:regularization}(b) gives \(\max_t e_{\mathrm{abs}}(t)=3.15\times10^{-9}\) and \(\max_t e_{\mathrm{rel}}(t)=2.12\times10^{-8}\).

These discrepancies remain at the level of the prescribed integration tolerances, providing a stringent numerical validation of the regularization identity.

The regularized evolution established in this subsection provides the starting point for the stability certification examined next.

\subsection{Compatibility Principle and Regularization}
\label{subsec:num-compatibility}

This subsection illustrates the compatibility principle established in Sections~\ref{sec:compatibility} and~\ref{sec:factorization}. The numerical experiments first examine the effect of compatibility defects and then validate the regularization mechanism obtained when the compatibility condition is exactly satisfied.

To examine the compatibility obstruction, the nominal compatible gain
\(
K_0
\)
is compared with the perturbed gains
\[
K_0^\delta
=
K_0+\delta I_2,
\qquad
\delta\in\{0.05,0.10,0.20\}.
\]
The corresponding projected compatibility defect is
\[
R_N
=
A_{1,N}-B_NK_0^\delta C_N.
\]

The resulting closed-loop trajectories are reported in Figure~\ref{fig:compatibility}(a). The compatible realization exhibits the fastest decay, whereas increasing values of \(\delta\) progressively amplify the initial transient. These trajectories should not be interpreted as indicating that every compatibility defect necessarily destroys stability. Rather, they illustrate the persistence of the singular contribution generated by a nonzero compatibility defect, even when the remaining dynamics remain sufficiently dissipative to preserve long-time decay.

The singular contribution
\[
\frac{\|R_Nz_N(t)\|_2}{\alpha(t)}
\]
is displayed in
Figure~\ref{fig:compatibility}(b). For the compatible realization,
\(
R_N=0,
\)
and the singular term vanishes identically. In contrast, the factor
\(
\frac{1}{\alpha(t)}
\)
produces a pronounced amplification near the degenerate initial time whenever the compatibility condition is violated.

The sensitivity to the lower integration threshold
\(
t_\varepsilon
\)
is illustrated in Figure~\ref{fig:compatibility}(c). At the observation time
\(
t_\star=0.25,
\)
the compatible trajectory remains essentially insensitive to variations of
\[
t_\varepsilon\in[10^{-4},10^{-2}],
\]
whereas the incompatible trajectories become increasingly sensitive as the degenerate time is approached. This behavior provides a numerical illustration of the structural obstruction established in
Section~\ref{sec:compatibility}.

\begin{figure}[!th]
\centering
\includegraphics[width=.95\textwidth]{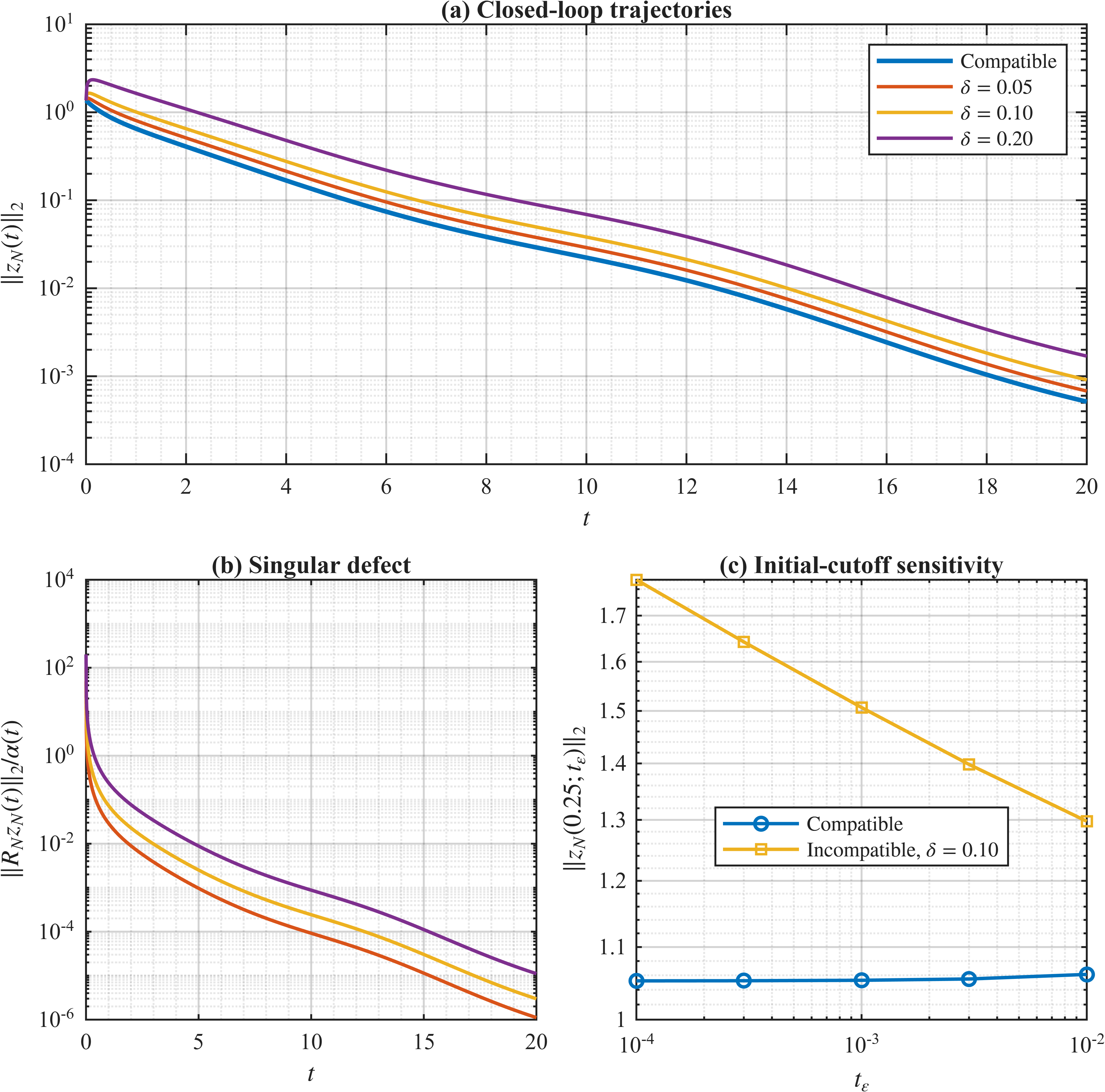}
\caption{Numerical assessment of the compatibility principle.
(a) Closed-loop trajectories for the compatible realization and for
three compatibility defects.
(b) Singular contribution \(\|R_Nz_N(t)\|_2/\alpha(t)\) in the incompatible cases.
(c) Sensitivity of the state at \(t_\star=0.25\) to the lower
integration threshold \(t_\varepsilon\).}
\label{fig:compatibility}
\end{figure}

Once compatibility is enforced, the singular closed-loop evolution reduces to its regularized counterpart. The corresponding numerical trajectories are compared in Figure~\ref{fig:regularization}(a), where the compatible degenerate and regularized solutions remain visually indistinguishable throughout the simulation interval.

Define
\[
e_{\mathrm{abs}}(t)
=
\|z_N^{\mathrm{deg}}(t)-z_N^{\mathrm{reg}}(t)\|_2
\]
and
\[
e_{\mathrm{rel}}(t)
=
\frac{
\|z_N^{\mathrm{deg}}(t)-z_N^{\mathrm{reg}}(t)\|_2
}{
\max\{\|z_N^{\mathrm{reg}}(t)\|_2,\varepsilon_{\mathrm{num}}\}
}.
\]

The trajectory discrepancies reported in
Figure~\ref{fig:regularization}(b) satisfy
\[
\max_t e_{\mathrm{abs}}(t)
=
3.15\times10^{-9},
\]
and
\[
\max_t e_{\mathrm{rel}}(t)
=
2.12\times10^{-8}.
\]

These discrepancies remain at the level of the prescribed numerical tolerances, providing a stringent validation of the regularization identity established in Section~\ref{sec:regular-evolution}.

\begin{figure}[!th]
\centering
\includegraphics[width=\columnwidth]{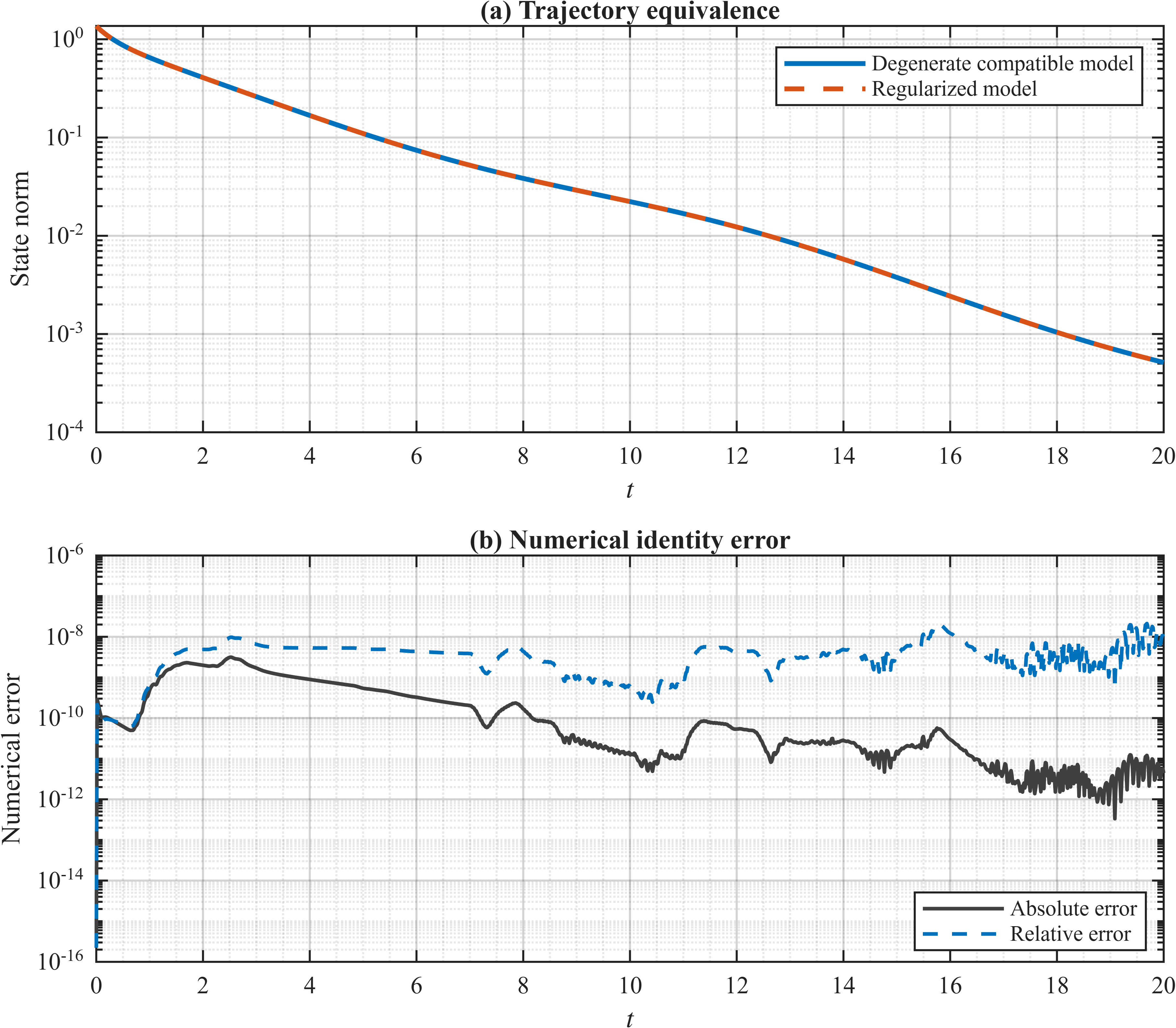}
\caption{Numerical verification of the regularization identity.
(a) Compatible degenerate and regularized trajectories.
(b) Absolute and relative trajectory discrepancies.}
\label{fig:regularization}
\end{figure}

The regularized evolution validated in this subsection provides the starting point for the stability certification examined next.

\subsection{Certified Critical--Residual Stabilization}
\label{subsec:num-certificate}

This subsection illustrates the uniform stability certificate developed in Section~\ref{sec:critical-residual}. The numerical experiments first verify the critical--residual certification at the minimal admissible spectral cutoff and then examine the asymptotic behavior of the certification quantities as the truncation order increases.

The feedback gain \(K_1\) is synthesized from the two unstable modes, whereas the critical--residual certificate is evaluated for increasing spectral cutoffs. The smallest cutoff satisfying the complete small-gain condition is \(n_\star=3\). Thus, the controller synthesis dimension remains equal to two, while one additional dissipative mode is incorporated into the finite-dimensional certificate.

At \(n_\star=3\), the computed certification constants are
\[
\mu_c=0.4068,
\qquad
\mu_s=2.8424,
\]
and
\[
\gamma_{cs}
=
\gamma_{sc}
=
1.006.
\]

The weight maximizing the algebraic small-gain margin is \(\eta_{\mathrm{margin}}=1.2871\), for which
\[
4\eta_{\mathrm{margin}}\mu_c\mu_s
-
\bigl(
\gamma_{cs}
+
\eta_{\mathrm{margin}}\gamma_{sc}
\bigr)^2
=
0.6592.
\]
The corresponding normalized margin equals \(0.1107\).

A second optimization, aimed at maximizing the certified decay rate, gives \(\eta_{\mathrm{rate}}=0.9982\), together with
\[
M=1.0009,
\qquad
\omega_{\mathrm{cert}}=0.0450.
\]
The resulting certificate is
\[
\|z_N(t)\|_2
\le
1.0009\,
e^{-0.0450(t-t_0)}
\|z_N(t_0)\|_2.
\]

For \(n\ge2\), one has \(Q_nA_rQ_n=0\). Moreover, since \(C=B^\ast\) and \(K_1\) is negative definite, \(-Q_nBK_1B^\ast Q_n\) is positive semidefinite in the dissipative operator. Consequently, the feedback does not reduce the residual diffusion margin, which justifies the estimate
\[
\mu_s=d_{\min}\lambda_{n+1}.
\]

The certified closed-loop responses are reported in Figure~\ref{fig:certificate}(a). Without the stabilizing component \(K_1\), the critical modes grow rapidly. Incorporating the proposed feedback restores exponential decay of the complete state.

Figure~\ref{fig:certificate}(b) displays the total, critical, and residual norms together with the certified exponential envelope. The observed asymptotic decay rate is \(\omega_{\mathrm{obs}}=0.3674\), which satisfies
\[
\frac{\omega_{\mathrm{obs}}}
{\omega_{\mathrm{cert}}}
=
8.16.
\]
The difference reflects the conservatism introduced by the uniform operator estimates and the critical--residual small-gain argument.
Nevertheless, the certified envelope remains above the computed trajectory throughout the simulation, in agreement with the theoretical certificate.

\begin{figure}[!th]
\centering
\includegraphics[width=\columnwidth]{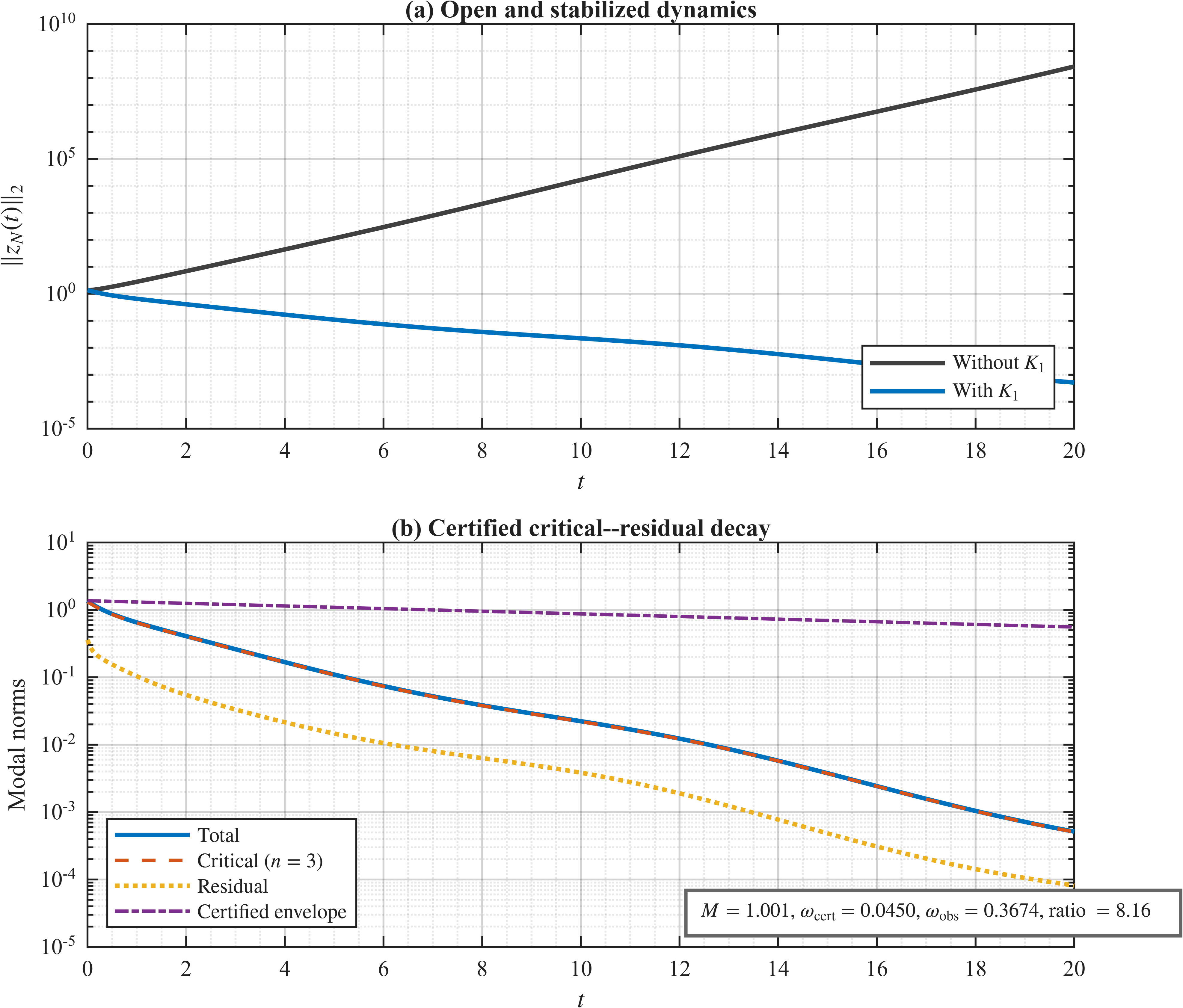}
\caption{Critical--residual stability certification.
(a) Regularized dynamics without and with the stabilizing gain \(K_1\).
(b) Total, critical, and residual state norms together with the certified exponential envelope.}
\label{fig:certificate}
\end{figure}

The structural mechanism underlying the stability certificate is illustrated in Figure~\ref{fig:spillover}. Figure~\ref{fig:spillover}(a) reports
\[
\gamma_{cs}(n)
=
\|P_nL_KQ_n\|,
\qquad
\gamma_{sc}(n)
=
\|Q_nL_KP_n\|,
\]
and
\[
\gamma_{ss}(n)
=
\|Q_nL_KQ_n\|,
\]
computed on the fixed reference approximation
\[
N_{\mathrm{ref}}=160.
\]
All three quantities decrease markedly as the spectral cutoff increases. In particular, \(\gamma_{ss}(n)\) decreases by more than seven orders of magnitude over the displayed range. The local plateaus originate from the symmetry of the distributed actuator profiles and do not alter the overall decay trend.

Figure~\ref{fig:spillover}(b) reports the normalized small-gain margin together with the largest certified decay rate. The margin first becomes positive at
\[
n=n_\star=3,
\]
which is precisely the smallest truncation order satisfying the theoretical certificate. The certified decay rate subsequently approaches a plateau, indicating that increasing the truncation order
beyond this point yields only marginal improvement of the guaranteed
stability estimate.

\begin{figure}[!th]
\centering
\includegraphics[width=\columnwidth]{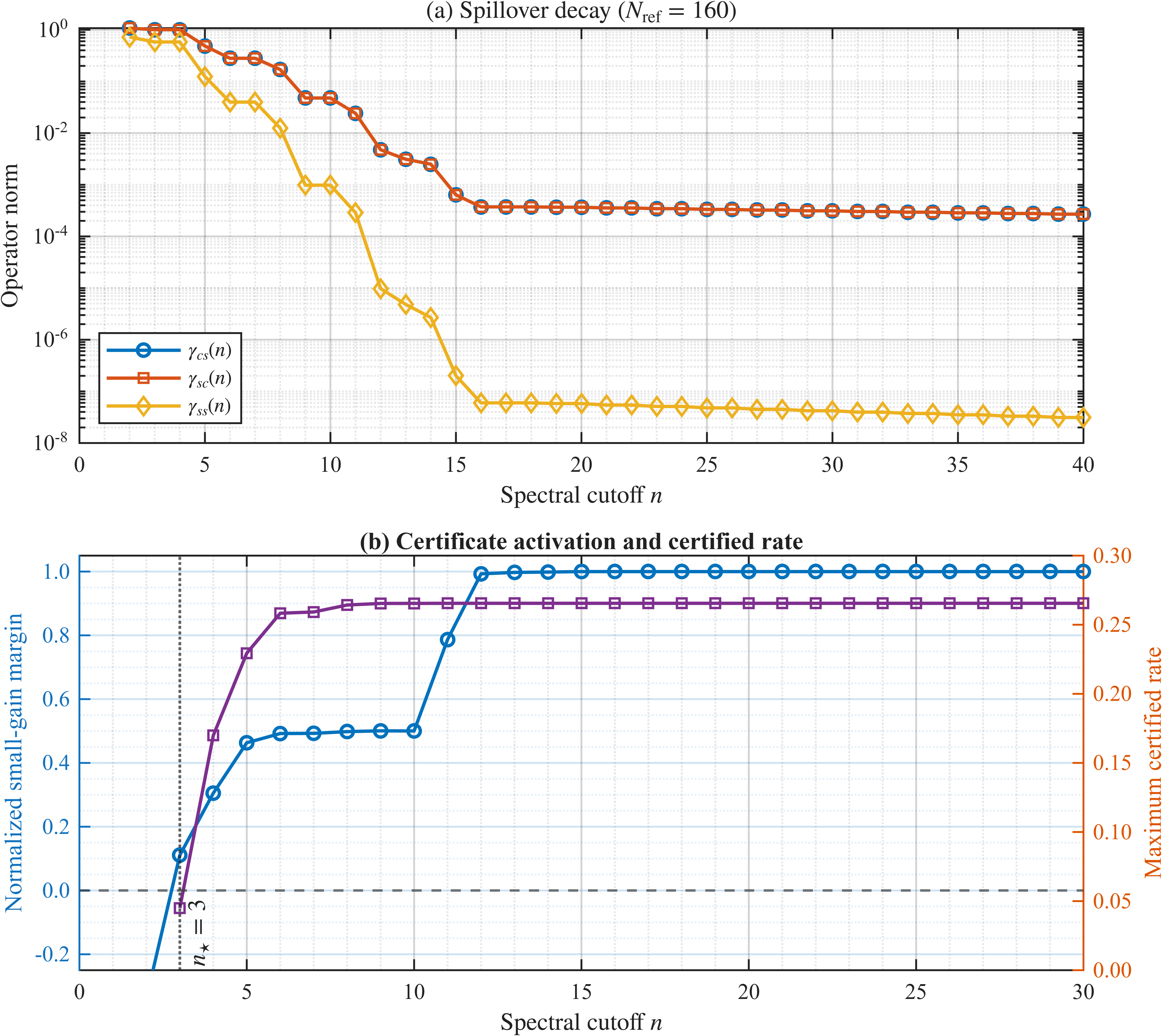}
\caption{Spectral certification mechanism.
(a) Decay of the critical--residual spillover norms on a fixed
high-order reference model.
(b) Normalized small-gain margin and maximum certified decay rate versus
the spectral cutoff.}
\label{fig:spillover}
\end{figure}

The numerical results are fully consistent with the uniform stability certificate established in Section~\ref{sec:critical-residual}. The
remaining numerical experiments illustrate how this finite-dimensional certification extends to the original infinite-dimensional evolution through the structural finite-to-infinite lifting.

\subsection{Finite-to-Infinite Lifting and Control Implementation}
\label{subsec:num-lifting}

This subsection illustrates the structural finite-to-infinite lifting
established in Section~\ref{sec:structural-lifting}. The numerical experiments first examine the robustness of the synthesized controller with respect to the Galerkin approximation order and then illustrate the resulting static output feedback implementation.

The finite-to-infinite lifting mechanism is examined by applying the
same feedback gain \(K_1\), synthesized from the first two modes, to
Galerkin approximations of orders
\[
N\in\{20,40,60,80\}.
\]
Since the critical matrices \(B_c\) and \(C_c\) are independent of the approximation order, the controller remains unchanged up to numerical roundoff:
\[
\max_N
\|K_1^{(N)}-K_1^{\mathrm{nom}}\|_2
<
10^{-15}.
\]

The corresponding closed-loop trajectories are reported in
Figure~\ref{fig:lifting}(a). They remain visually indistinguishable over the entire simulation interval despite the increasing approximation dimension.

The maximum relative discrepancies between successive Galerkin
approximations are shown in Figure~\ref{fig:lifting}(b) and satisfy
\[
1.63\times10^{-8},
\qquad
1.11\times10^{-8},
\qquad
8.58\times10^{-9}.
\]

These discrepancies are already at the level of the prescribed temporal integration accuracy for \(N=20\), indicating that the numerical solution is essentially insensitive to further increases in the approximation order. This observation provides a numerical illustration of the finite-to-infinite lifting mechanism without being interpreted as an independent proof of infinite-dimensional convergence.

\begin{figure}[!th]
\centering
\includegraphics[width=\columnwidth]{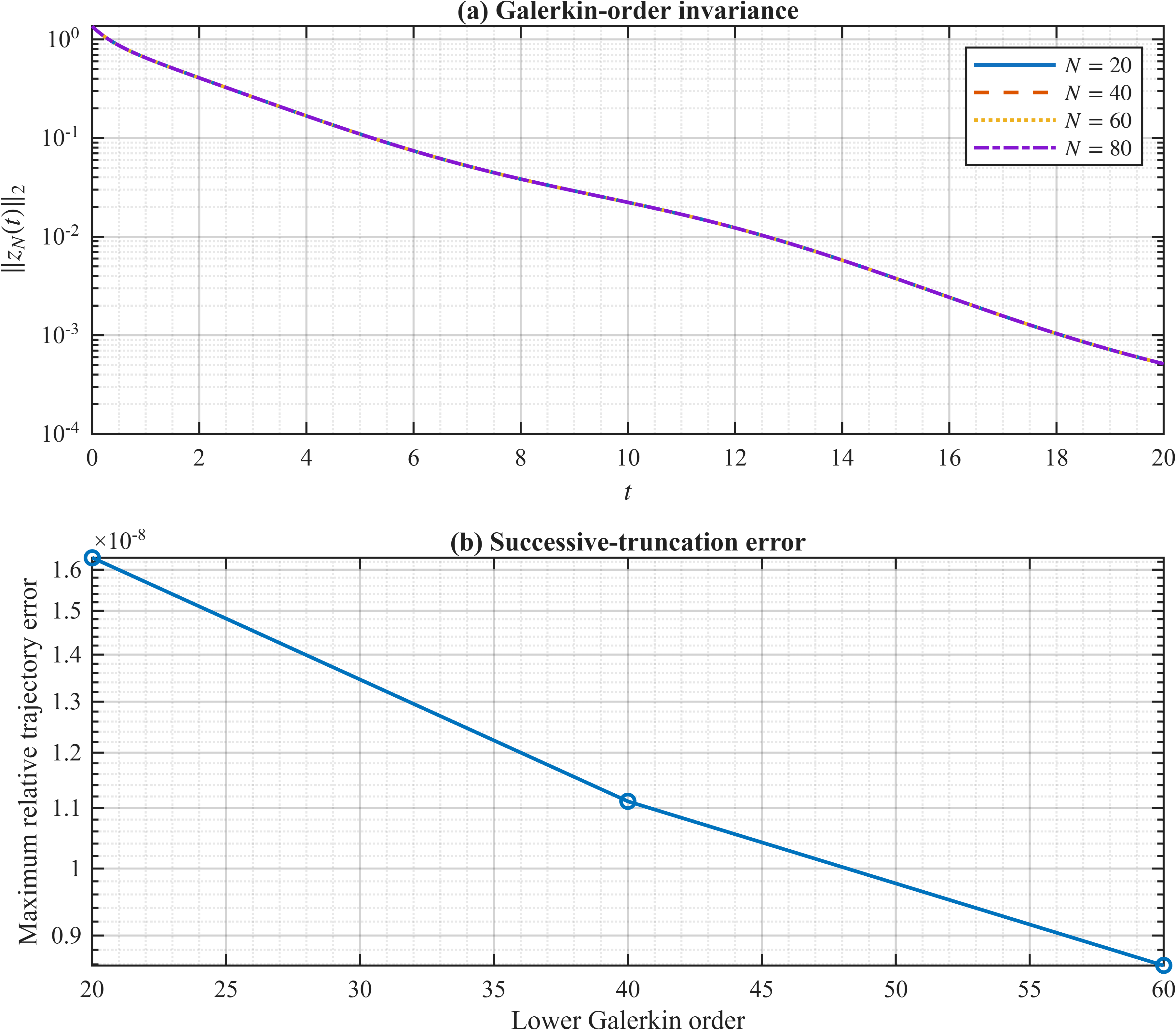}
\caption{Numerical assessment of the finite-to-infinite lifting.
(a) Closed-loop trajectories for four Galerkin orders using the same
critical controller.
(b) Maximum relative discrepancies between successive truncations.}
\label{fig:lifting}
\end{figure}

The implementation of the resulting static output feedback law is
illustrated next. Recall that the control input is given by
\[
u(t)
=
K_0y(t)
+
\alpha(t)K_1y(t).
\]

The total control norm together with its compatibility and stabilizing components is displayed in Figure~\ref{fig:control}(a). In the immediate neighborhood of the degenerate initial time, the stabilizing contribution is naturally attenuated by the factor
\(
\alpha(t),
\)
so that the compatibility component dominates the control action. As the regularized dynamics evolve, the stabilizing component rapidly becomes dominant before both contributions decay together with the state.

The control effort remains uniformly bounded throughout the simulation, with
\[
\max_{t\in[10^{-3},20]}
\|u(t)\|_2
=
1.1451.
\]
The individual actuator commands are reported in Figure~\ref{fig:control}(b). Both inputs remain smooth over the entire simulation interval and converge asymptotically to zero, confirming the practical implementability of the proposed static output feedback law without requiring any singular control effort near
\(
t=0.
\)

\begin{figure}[H]
\centering
\includegraphics[width=\columnwidth]{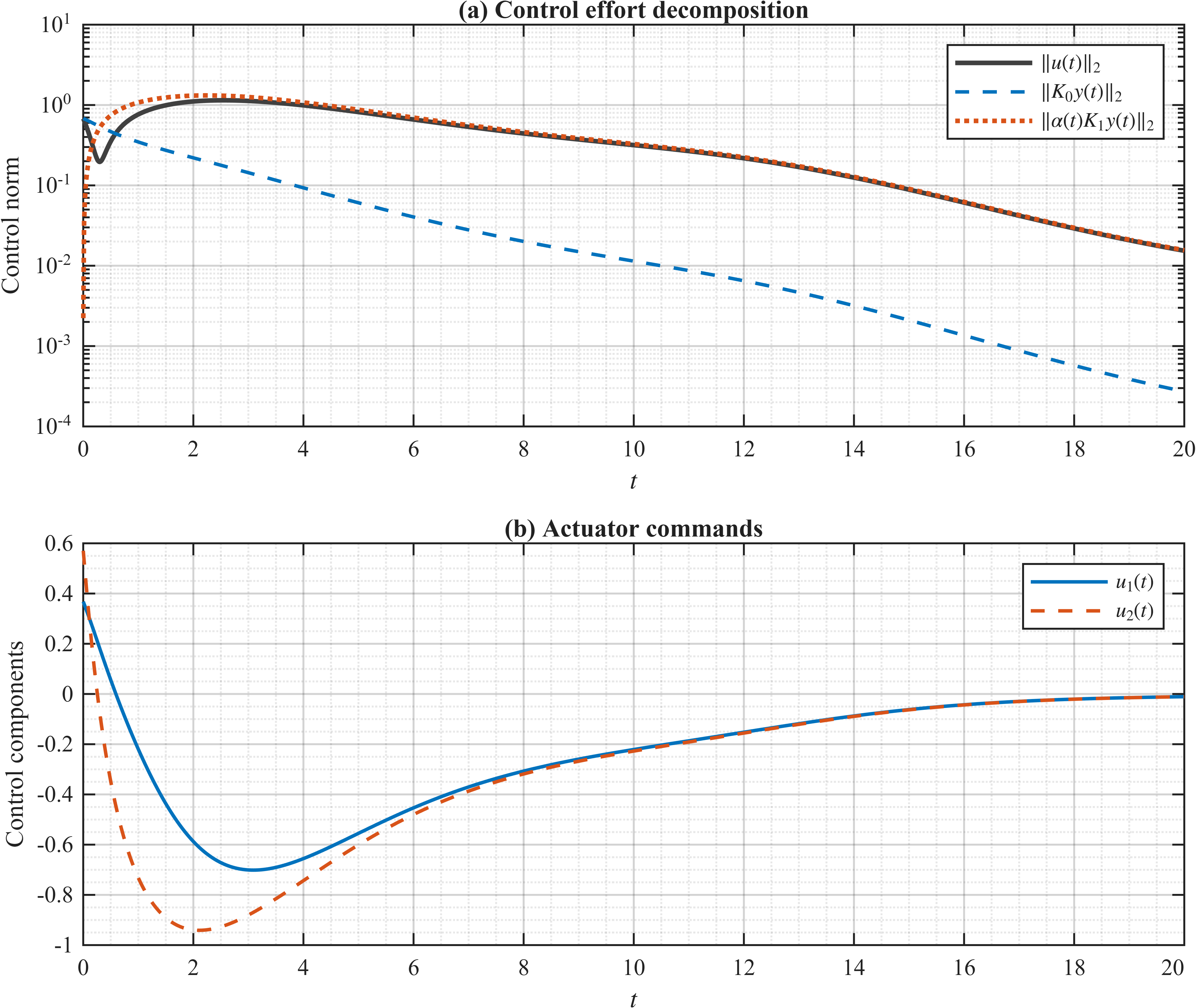}
\caption{Static output feedback implementation.
(a) Total control effort and decomposition into compatibility and
stabilizing components.
(b) Distributed actuator commands.}
\label{fig:control}
\end{figure}

The numerical experiments are fully consistent with the theoretical
developments established throughout the paper. They illustrate the
compatibility principle, validate the regularization mechanism, confirm the critical--residual stability certificate, demonstrate the structural finite-to-infinite lifting, and show that the resulting static output-feedback controller can be implemented without singular control effort.

\section{Conclusion}
\label{sec:conclusion}

The main outcome of this work is the identification of structural admissibility as the central issue underlying feedback design for temporally degenerate parabolic systems. For this class of systems, temporal degeneracy does not merely complicate stability analysis; it modifies the problem itself by making the existence of a mathematically admissible closed-loop evolution a prerequisite to any subsequent control objective.

This change of perspective explains the overall organization of the theory. The compatibility principle removes the structural obstruction introduced by temporal degeneracy, thereby restoring a regular closed-loop evolution. Stability certification, finite-to-infinite lifting, and feedback synthesis then arise as successive consequences of this structural resolution, leading to a complete characterization of the class of temporally degenerate linear parabolic systems considered here.

The resulting theory also delineates its own natural boundaries. Extending the compatibility principle beyond exact compatibility, establishing analogous structural conditions for nonlinear degenerate systems, and exploiting compatibility as a guideline for actuator--sensor co-design remain significant open questions. These challenges do not reflect shortcomings of the present analysis; rather, they define the next stage in the development of the theory.

More broadly, the results support the view that structural admissibility is not simply a technical issue associated with degenerate dynamics, but a fundamental notion for the analysis of singular infinite-dimensional control systems. For such systems, establishing the existence of an admissible closed-loop evolution may be as fundamental as establishing its stability.

\section*{Funding}
There was no funding for the paper.





\end{document}